\documentclass[]{article}
\usepackage{amsmath,amsthm}
\usepackage{amssymb}
\usepackage{mathrsfs}

\usepackage{tikz-cd}
\newtheorem{thm}{Theorem}[section]
\newtheorem{cor}[thm]{Corollary}
\newtheorem{lem}[thm]{Lemma}
\newtheorem*{mthm}{Main Theorem}
\newtheorem{prop}[thm]{Proposition}

\theoremstyle{definition}
\newtheorem{exa}[thm]{Example}
\theoremstyle{definition}
\newtheorem{defn}[thm]{Definition}
\theoremstyle{remark}
\newtheorem{rem}[thm]{Remark}

\DeclareMathOperator{\spec}{spec}

\newcommand{\E}{\mathcal{E}}

\makeatletter
\newcommand{\eqnum}{\refstepcounter{equation}\textup{\tagform@{\theequation}}}
\makeatother
\newcounter{ale}

\title{Half-centered Operators}
\author{Olof Giselsson}

\begin{document}
\maketitle
\begin{abstract}
	An operator $T$ on a Hilbert space is called half-centered if the sequence $T^{*}T,(T^{*})^{2}T^{2},...$ consists of mutually commuting operators. It is a subclass of the well-studied centered operators. In this paper we give a condition for when a half-centered operator is centered and prove a structure theorem for those half-centered operators that satisfies a criteria of a technical nature.
\end{abstract}

\section{Introduction}
	A bounded operator $T$ on Hilbert space $\mathcal{H}$ is called \textit{centered} if the operators in the sequence 
	\begin{equation}\label{cente}
		\left\{ ...T^{3}T^{*3},T^{2}T^{*2},TT^{*},T^{*}T,T^{*2}T^{2},T^{*3}T^{3}... \right\}
	\end{equation}
	are mutually commuting. Examples includes weighted shifts and obviously isometries and self-adjoint operators. The structure of these operators is well understood; it has been shown in~\cite{ARS} that, a bit simplified, a general centered operator is a direct sum of weighted shifts (unilateral, bilateral or truncated). Another interesting article on the subject is~\cite{vp}, here some particular centered operators are investigated in relation to more general problems in operator theory.
	\\
	
	The purpose of this paper is to investigate operators satisfying the more general condition that the sequence 
	
	\begin{equation}\label{half}
		\left\{ T^{*}T,T^{*2}T^{2},T^{*3}T^{3}... \right\}
	\end{equation}
	consists of mutually commuting operators. As~\eqref{half} is half of~\eqref{cente}, we call such an operator \textit{half-centered}.
	\\
	
	We will mainly consider half-centered operators satisfying 
	\\
	$\dim(T\mathcal{H})^{\bot} = 1$ and a certain technical density criteria, which is not too restrictive. It turns out, under these assumptions, that either the structure of $T$ is very simple and can be explicitly described, or else the operators in the sequence $\left\lbrace T^{*k}T^{k}, k\in \mathbb{N}\right\rbrace $ are not linearly independent. Specifically, there exists $a,b,c,d\in\mathbb{R}$ not all zero, and $n,m\in\mathbb{N}^{+}$ such that 
	\begin{equation}\label{mad}
		a I +b T^{*n}T^{n}+ c T^{*m}T^{m}+ d T^{*m+n}T^{m+n}=0
	\end{equation}
	
	This is the main result here and most of the text is concerned with proving this.
	\\
	
	In section $2,$ we first prove a result that gives a necessary and sufficient criteria for when a half-centered operator is centered; for example, any half-centered operator with dense range is centered.
	We will then give several examples of classes of half-centered operators that are not necessarily centered, some of which has been extensively studied in the litterature.
	It will also be shown that some very natural operators are half-centered.
	For instance, any operator $T\in\mathcal{B}(L^{2}(X,\mu)),$ that acts by $f(x)\mapsto a(x)f(\phi(x))$ where $a\in L^{\infty}(X,\mu)$ and $\phi:X\rightarrow X$ is a measurable function, is half-centered by Proposition~\ref{exempl}.
	We will here also state the main theorem and discuss the conditions under which it holds. 
	\\
	
	This text is written in a decreasing level of generality.
	In section $3$ we will develop a theory for general injective operators   
	that is needed in the later sections and which provides a useful framework to analyze the half-centered operators. Here we will also prove some more general results about half-centered operators which do not necessarily fall under the hypothesis of the main theorem.
	\\
	
	Section $4$ concerns injective half-centered operators $T$ with $\dim (\ker T^{*})=1.$ It will be shown that in this case, the spectrum of $T^{*k}T^{k}$ restricted to certain subspaces can be quite effectively analyzed. 
	\\
	
	In the last sections, $5$ and $6,$ we will include then density condition as an assumption and prove the main result. 
	\\
	
	Acknowledgement: The author likes to thank \r{A}se Fahlander, Alexandru Aleman and Lyudmila Turowska.

	\section{Half-centered Operators: examples and notations}
		
		Clearly every centered operator is half-centered.
		As a first basic result, we give a characterization of the half-centered operators that are also centered. Here we use the notation $\E:=\ker T^{*}= (T\mathcal{H})^{\bot}.$ For a closed subspace $V\subseteq \mathcal{H},$ denote by $P_{V}$ the orthogonal projection onto $V.$
		\begin{prop}\label{surject}
			Let $T\in \mathcal{B(H)}$ be a half-centered operator. The following are equivalent:
			\begin{enumerate}{\Roman{enumi}}
				\item $T$ is centered.
				\item $T^{k*}T^{k}\E\subseteq \E$ for all $k\in\mathbb{N}.$
			\end{enumerate}
			
		\end{prop}
		
		\begin{proof}
			(\textit{1}$\Rightarrow$\textit{2}). Since $(TT^{*})(T^{k*}T^{k})=(T^{k*}T^{k})(TT^{*})$ for all $k\in \mathbb{N},$ it is easy to see that the space $\E=\ker T^{*}=\ker T T^{*}$ is invariant under the operators $T^{*k}T^{k}.$
			\\
			
			(\textit{2}$\Rightarrow$\textit{1}). First, we notice that by \textit{2}, the projection $P_{\E}$ commutes with the operators $T^{*k}T^{k},$ this is then also true for the projection $P_{\overline{T\mathcal{H}}}=I-P_{\E}.$
			\\
			
			Now since $$T^{*}((T^{*j}T^{j})(T T^{*}))T=(T^{*(j+1)}T^{j+1})(T^{*}T)=$$
			$$(T^{*}T)(T^{*(j+1)}T^{j+1})=T^{*}((T T^{*})(T^{*j}T^{j}))T,$$  we have $(P_{\overline{T\mathcal{H}}}(T^{*j}T^{j})P_{\overline{T\mathcal{H}}})(T T^{*})=(T T^{*})(P_{\overline{T\mathcal{H}}}(T^{*j}T^{j})P_{\overline{T\mathcal{H}}})$ for all $j\in\mathbb{N}.$
			This gives 
			\begin{equation}\label{ecco}
				[T^{*j}T^{j},T T^{*}]=P_{\E}(T^{*j}T^{j})TT^{*}-TT^{*}(T^{*j}T^{j})P_{\E}=0
			\end{equation}
			as $TT^{*}(T^{*j}T^{j})P_{\E}=TT^{*}P_{\overline{T\mathcal{H}}}(T^{*j}T^{j})P_{\E}=0$ by \textit{2}, so $T^{*k}T^{k}$ commutes with $TT^{*}.$
			\\
			
			Hence for any $k\in\mathbb{N},$ we have
			
			\begin{equation}\label{millia}
				\begin{split}
					& T^{*k}((T^{*j}T^{j})(T^{k+1}T^{*(k+1)}))T^{k}=(T^{*(j+k)}T^{j+k})(T T^{*})(T^{*k}T^{k})=\\
					& (T^{*k}T^{k})(TT^{*})(T^{*(j+k)}T^{j+k})=T^{*k}((T^{k+1}T^{*(k+1)})(T^{*j}T^{j}))T^{k},
				\end{split}
			\end{equation}
			where the second equality follows from~\eqref{ecco}. As $$P_{\overline{T^{k}\mathcal{H}}}(T^{k+1}T^{*(k+1)})P_{\overline{T^{k}\mathcal{H}}}=T^{k+1}T^{*(k+1)},$$ we get from~\eqref{millia}
			\begin{equation}\label{lemmel}
				(P_{\overline{T^{k}\mathcal{H}}}(T^{*j}T^{j})P_{\overline{T^{k}\mathcal{H}}})(T^{k+1}T^{*(k+1)})=(T^{k+1}T^{*(k+1)})(P_{\overline{T^{k}\mathcal{H}}}(T^{*j}T^{j})P_{\overline{T^{k}\mathcal{H}}})
			\end{equation}
			for all $k,j\in \mathbb{N}.$ We claim that~\eqref{lemmel} actually implies
			\begin{equation}\label{vovve}
				(T^{k}T^{*k})(T^{*j}T^{j})=(T^{*j}T^{j})(T^{k}T^{*k})\text{ for all $j,k\in\mathbb{N}$}.
			\end{equation}  
			The proof is by induction on $k.$ We already know that it holds for $k=1,$ so assume it is true for $k-1\geq 1.$ Now $(T^{k-1}T^{*(k-1)})(T^{*j}T^{j})=(T^{*j}T^{j})(T^{k-1}T^{*(k-1)})$ gives 
			$(T^{*j}T^{j})P_{\overline{T^{k-1}\mathcal{H}}}=P_{\overline{T^{k-1}\mathcal{H}}}(T^{*j}T^{j})$ since $P_{\overline{T^{k-1}\mathcal{H}}}=I-P_{\ker (T^{k-1}T^{*(k-1)})} .$ As $P_{\overline{T^{k-1}\mathcal{H}}}(T^{k}T^{*k})=(T^{k}T^{*k})$ we see that $$T^{*j}T^{j}T^{k}T^{*k}=(P_{\overline{T^{k-1}\mathcal{H}}}T^{*j}T^{j}P_{\overline{T^{k-1}\mathcal{H}}})T^{k}T^{*k}$$
			$$T^{k}T^{*k}T^{*j}T^{j}=T^{k}T^{*k}(P_{\overline{T^{k-1}\mathcal{H}}}T^{*j}T^{j}P_{\overline{T^{k-1}\mathcal{H}}})$$ and by~\eqref{lemmel}, the right hand sides are equal. Hence~\eqref{lemmel} is true also for $k.$
			\\
			
			There is only the equality $(T^{k}T^{*k})(T^{m}T^{*m})=(T^{m}T^{*m})(T^{k}T^{*k})$ left to prove. But this follows from what has already been proven, since if, say $m\geq k,$ then $$(T^{k}T^{*k})(T^{m}T^{*m})=T^{k}((T^{*k}T^{k})(T^{m-k}T^{*(m-k)}))T^{*k}=$$
			$$T^{k}((T^{m-k}T^{*(m-k)})(T^{*k}T^{k}))T^{*k}=(T^{m}T^{*m})(T^{k}T^{*k}).$$ The proof is now complete.

		\end{proof}
		
		\begin{cor}
			If $T\in\mathcal{B(H)}$ is half-centered and $\overline{T\mathcal{H}}=\mathcal{H},$ then $T$ is centered.
		\end{cor}
		As it was already noted, any centered operator is half-centered. But we will see below that there are a lots of half-centered operators that are not centered.
		
		\begin{exa}($2$-isometries)
			An operator $T$ satisfying the equation
			\begin{equation}\label{2iso}
				I-2T^{*}T+T^{2*}T^{2}=0
			\end{equation}
			is called a $2$-isometry. The equation~\eqref{2iso} implies that $T^{*k}T^{k}$ is a linear combination of $I$ and $T^{*}T$ for every $k\in\mathbb{N}$ and this gives that $T$ is half-centered.
			$2$-isometries has been studied a lot due to their connection with the Dirichlet shift (see~\cite{CO1}~\cite{CO2}~\cite{CO3}). From their theory one can deduce that a centered $2$-isometry must be in the form $T=U\oplus S,$ with $U$ an isometry and $S$ a weighted shift. In general, $2$-isometries has a quite complicated structure, so in this case the centered $2$-isometries forms a strict (and quite boring) subclass.
		\end{exa}
		
		More generally, any operator $T$ satisfying
		
		$$a_{0}I+a_{1}T^{*}T+a_{2}T^{*2}T^{2}=0$$ for constants $a_{0},a_{1},a_{2}\in \mathbb{R}$ (where at least one $a_{i}\neq 0$) will be half-centered,
		since then again every $T^{*k}T^{k}$ will be a linear combination of $I$ and $T^{*}T.$

		\begin{exa}
			Let $P,Q$ be two orthogonal projections and consider $$T=PQ.$$
			Then $T$ is half-centered since $$T^{*k}T^{k}=\prod_{j=1}^{k}QP\prod_{j=1}^{k}PQ=Q\prod_{j=1}^{2k-1}PQ=QT^{2k-1}$$ and so $$\left(T^{*j}T^{j}\right)\left(T^{*k}T^{k}\right)=QT^{2j-1}QT^{2k-1}=QT^{2k+2j-2}=$$
			$$QT^{2k-1}QT^{2j-1}=\left(T^{*k}T^{k}\right)\left(T^{*j}T^{j}\right).$$ 
			Now, $TT^{*}=PQP$ and from this we calculate $$\left(TT^{*}\right)\left(T^{*}T\right)=PQPQPQ=T^{3}$$ $$\left(T^{*}T\right)\left(TT^{*}\right) =QPQPQP=T^{*3}.$$ So if $T^{*3}\neq T^{3}$ then $T$ is half-centered but not centered.
			The latter holds if we take, for example $$P=\left[ \begin{array}{ccc}
			\frac{1}{2} & -\frac{1}{2}  \\
			-\frac{1}{2} & \frac{1}{2}  
			\end{array} \right],Q=\left[ \begin{array}{ccc}
			1 & 0  \\
			0 & 0  
			\end{array} \right]$$
			then $T=\left[ \begin{array}{ccc}
			\frac{1}{2} & 0 \\
			-\frac{1}{2} & 0  
			\end{array} \right]$ and so $$T^{3}=\left[ \begin{array}{ccc}
			\frac{1}{2^{3}} & 0 \\
			-\frac{1}{2^{3}} & 0  
			\end{array} \right]\neq \left[ \begin{array}{ccc}
			\frac{1}{2^{3}} & -\frac{1}{2^{3}} \\
			0 & 0  
			\end{array} \right]=T^{*3}.$$
		\end{exa}
		A large class of half-centered operators are given by the following proposition:
		\begin{prop}\label{exempl}
			Let $(X,\mu)$ be a measure space with $\sigma$-finite measure $\mu$ and let $\psi:X\rightarrow X$ be a measurable function such that the linear map $$f(x) \mapsto f(\psi(x))$$ induces a bounded linear operator on $L^{2}(X,\mu)$ and let $\xi\in L^{\infty}(X,\mu).$ Then the operator 
			\begin{equation}\label{wqqw}	
				T:f\in L^{2}(X,\mu)\mapsto \xi(x)f(\psi(x)) 
			\end{equation}
			is half-centered.
		\end{prop}
		
\begin{proof}
Since $T$ is the composition of two bounded linear operators, we have $T\in \mathcal{B}(L^{2}(X,\mu)).$ Take any $h\in L^{\infty}(X,\mu)$ and let $M_{h}$ denote multiplication by $h.$ For all $f,g \in L^{2}(X,\mu)$
$$\left\langle  g,T^{*} T M_{h}f \right\rangle=\left\langle T g, T M_{h}f \right\rangle= \int_{X}\overline{T(g)(x)}T(h f)(x)d\mu(x)=$$ $$\int_{X}\overline{\xi(x)g(\psi(x)) }\xi(x)h(\psi(x))f(\psi(x)) d\mu(x)=$$
$$\int_{X}\overline{\xi(x)\overline{h(\psi(x))}g(\psi(x)) }\xi(x)f(\psi(x)) d\mu(x)=\int_{X}\overline{T(\overline{h}g)(x)}T(f)(x)d\mu(x)=$$
$$\left\langle T  M_{\bar{h}}g, Tf \right\rangle=\left\langle T  M_{h}^{*}g, Tf \right\rangle=\left\langle  g,M_{h}T^{*} T f \right\rangle.$$
This gives $M_{h}T^{*} T=T^{*} TM_{h}$ and since $h$ was arbitrary $T^{*}T$ commutes with all of $L^{\infty}(X,\mu).$
The von Neumann algebra $L^{\infty}(X,\mu)\subseteq \mathcal{B}(L^{2}(X,\mu))$ is maximal abelian (see~\cite{muphy}) and so $T^{*} T\in L^{\infty}(X,\mu).$
The same argument gives $T^{*k}T^{k}\in L^{\infty}(X,\mu)$ for every $k\in\mathbb{N}$ and therefore the operators in the sequence~\eqref{half} commute with each other. 
\end{proof}
Notice that if, for example, the set $\left\lbrace \xi(x)f(\psi(x)); f\in L^{2}(X,\mu)\right\rbrace $ in is dense in $L^{2}(X,\mu)$ then Proposition~\ref{surject} gives that $T$ is actually centered.
However, in general the operators defined in Proposition~\ref{exempl} will not be centered.
\\
		
Before we proceed any further, let's first fix some notations.
		\\
		
The operators $T^{*k}T^{k}$ are referred to a lot, so in order to make things appear more concise, we write them as $T_{k}$. We again denote $\ker T^{*}=(T\mathcal H)^{\bot}$ by $\mathcal{E}$ and this notation will be used for the rest of the text.
		\\
		
Next, we define a subspace that will be of utmost importance here:
		\\
		
Let $\mathcal{M_{E}}$ be the smallest closed subspace containing $\E$ that is invariant with respect to all the operators $T_{k}.$
Proposition~\ref{surject} indicates that $\mathcal{M_{E}}$ is a natural starting point when investigating the strictly half-centered operators, since this result is saying that half-centered $T$ is centered iff $\mathcal{M_{E}}=\E.$
\\
		
If we have an operator $R\in \mathcal{B(H)}$ and a closed subspace $V\subseteq \mathcal{H}$ such that $R V\subseteq V$ and $R^{*}V\subseteq V,$ then $V$ is said to be \emph{reducing} for $R.$
In the case when $R$ has no reducing subspaces, $R$ is called \emph{irreducible}. If $T$ is centered, then $\E$ is a reducing subspace for both $T_{k}$ and $T^{k}T^{*k}.$ Assuming $\E\neq 0,$ then if $T$ is centered and irreducible, we must have $\dim\ker T^{*}=1.$ This is generally not true for half-centered operators.
\\
		
In this paper we will prove a structure theorem for a half-centered operator satisfying the following assumptions:
\renewcommand{\labelenumi}{\textbf{\Roman{enumi}}}
\begin{enumerate}{\Roman{enumi}}
	\item $T$ is injective and $\mathcal{E}$ has dimension $1$.
			
	\item  $\bigvee_{k=0}^{\infty}T^{k}\mathcal{M_{E}}=\mathcal{H}.$
			
\end{enumerate}
		
Theorem~\ref{space1} below shows that $\bigvee_{k=0}^{\infty}T^{k}\mathcal{M_{E}}$ can alternately be defined as the smallest closed subspace containing $\E$ that is invariant under $T$ and the operators $T_{k}.$ However, without any further conditions this subspace will in general not be reducing $T.$
Notice also that these conditions implies that the Hilbert space $\mathcal{H}$ is separable. 
\\
		
Spread throughout the rest of this section are some examples of half-centered operators that satisfies conditions \textbf{I} and \textbf{II}.
\\
		
Let us recall the notion of "\emph{wandering subspace property}" for an injective operator $R$ on a Hilbert space $\mathcal{H}$. Given $R\in \mathcal{B(H)},$ let as before $\E:=ker R^{*},$ then $R$ is said to satisfy the "\emph{wandering subspace property}" if
\begin{equation}\label{wsp}
	\bigvee_{k=0}^{\infty}R^{k} \E=\mathcal{H}
\end{equation}
This condition resembles \textbf{II}. The subspace $\E$ is often called the wandering subspace. 
\\
		
Closely related to~\eqref{wsp} is the condition
\begin{equation}\label{wsop}
\bigcap_{k=0}^{\infty}R^{k}\mathcal{H}=\left\{0\right\}.
\end{equation}

If for an injective operator $R$ with closed range we let $R'=R(R^{*}R)^{-1},$ then by results in~\cite{SS}~\eqref{wsp} holds for $R$ iff~\eqref{wsop} holds for $R.$ Observe that $\ker R^{*}=\ker R'^{*}$ and $(R')'=R.$ The operator $R'$ is called the Cauchy dual of $R.$ 
\\
		
An important fact about injective operators satisfying~\eqref{wsp} and have closed range is that they are unitary equivalent to the multiplication operator $f(z)\mapsto z f(z)$ on a Hilbert space $\mathcal{L}(\E)$ of $\E$-valued analytic functions (with $\E=\ker R^{*}$).
\\
		
The condition \textbf{II} is actually weaker than~\eqref{wsp} for both $T$ and $T'$ since the subspace $\bigvee_{j=0}^{k}T^{j}\mathcal{M_{E}}$ contains both $T^{k}\mathcal{E}$ and $T'^{k}\mathcal{E}$ for every $k\in \mathbb{N}.$ Indeed, this is trivial for $T.$ To prove it for $T',$ notice that $T^{*}$ is a left inverse for $T',$ so that $T^{*k+1}T'^{k}\mathcal{E}=T^{*}\mathcal{E}=0.$ This gives $T'^{k}\mathcal{E}\subseteq \ker T^{*k+1}.$ It is not hard to see that $\ker T^{*k+1}$ is spanned by the subspaces $T^{j}(T^{*j}T^{j})^{-1}\E$ for $0\leq j\leq k$ and these are all subspaces of $\bigvee_{j=0}^{k}T^{j}\mathcal{M_{E}}.$ It now follows: 
		
\begin{prop}
\textbf{II} holds for $R$ if~\eqref{wsp} or~\eqref{wsop} holds for $R$ or $R'$ (if the latter operator exists).
\end{prop}
For instance, this implies that if $S$ is the shift operator on $\ell^{2},$ then as any operator of the form $A S A^{-1}$ with $A\in \mathcal{B}(\ell^{2})$ satisfies~\eqref{wsop}, it has property \textbf{II}.
\\
		
As two of the most distinguished cases of half-centered operators satisfying \textbf{I} and \textbf{II} are the weighted shifts and the $2$-isometries (in the irreducible non-isometry case) and both of these classes of operators satisfies~\eqref{wsp} and~\eqref{wsop} (this claim is trivial for weigthed shift. For $2$-isometries, see~\cite{SS}). It is natural to ask if~\eqref{wsp} and~\eqref{wsop} are true in general for a half-centered operator satisfing \textbf{I} and \textbf{II}. However, as our next example shows, this is not the case.
		
\begin{exa}\label{exampel}
Let $S$ be the isometric shift on the Hardy space $\mathbb{H}^{2}$, i.e $$f(z)\in\mathbb{H}^{2}\mapsto zf(z).$$ Now consider $$T=a S+ \left(I-S S^{*}\right),$$ with an $a\in\mathbb{C}$ such that $0<\left|a\right|< 1.$ An easy way to see that both $T$ and $T'=T \left(T^{*}T\right)^{-1}$ are half-centered is to write them down as matrices in the standard basis $\left\lbrace z^{k};k\in\mathbb{N}\right\rbrace.$ 
\begin{equation}\label{ex1}
	T=\left[ \begin{matrix}
		1 & 0 & 0 & 0 & .. \\
		a & 0 & 0 & 0 & .. \\
		0 & a & 0 & 0 & .. \\
		0 & 0 & a & 0 & .. \\
		.. & .. & .. & .. & .. \\
       \end{matrix} \right]
\end{equation}\label{ex2}
\begin{equation}
	T^{*}T=\left[ \begin{matrix}
		1+|a|^{2} & 0 & 0 & 0 & .. \\
		0 & |a|^{2} & 0 & 0 & .. \\
		0 & 0 & |a|^{2} & 0 & .. \\
		0 & 0 & 0 & |a|^{2} & .. \\
		.. & .. & .. & .. & .. \\
		\end{matrix} \right]
\end{equation}
\begin{equation}\label{ex3}
	T'=\left[ \begin{matrix}
	\frac{1}{1+\left|a\right|^{2}} & 0 & 0 & 0 & .. \\
	\frac{a}{1+\left|a\right|^{2}} & 0 & 0 & 0 & .. \\
	0 & \frac{a}{\left|a\right|^{2}} & 0 & 0 & .. \\
	0 & 0 & \frac{a}{\left|a\right|^{2}} & 0 & .. \\
		.. & .. & .. & .. & .. \\
	\end{matrix} \right].
\end{equation}
It is not hard to see now that for all $k\in\mathbb{N},$ both matrices $T^{*k}T^{k}$ and $T'^{*k}T'^{k}$ are diagonal.
From~\eqref{ex1}, we see that $\ker T^{*}$ is spanned by $\bar{a}- z $ and from~\eqref{ex2} that $T^{*}T-\left|a\right|^{2}I $ is the operator $f(z)\mapsto (1+\left|a\right|^{2})f(0).$ Thus $1,z\in\mathcal{M_{E}}$ and since $T^{k}z=a^{k}z^{k},$ this gives $$\bigvee_{k=0}^{\infty}T^{k} \mathcal{M_{E}}=\mathbb{H}^{2}.$$ Hence both \textbf{I} and \textbf{II} are fulfilled by $T$.
However, as $$\sum_{k=0}^{\infty}a^{k}z^{k}=\frac{1}{a- z}\in \mathbb{H}^{2}$$ is an eigenvector for $T$ and thus in the range of $T^{k}$ for all $k\in \mathbb{N},$ $T$ does not satisfy~\eqref{wsop} and hence the Cauchy dual $T'$ does not possess the wandering subspace property.
\end{exa}
		
\begin{exa}\label{llll}
The operator in Example~\ref{exampel} is a special case of a more general type of half-centered operator.
Let $\mathcal{H}$ be a separable Hilbert space with an orthonormal basis $\left\{x_{k}:k\in\mathbb{N}\right\}$ and inner product $\left\langle .,.\right\rangle.$ Let $J$ be an injective shift operator with respect to this basis, so that $$J x_{k}=a_{k} x_{k+1}$$ for some nonzero constants $a_{k}\in\mathbb{C}.$
If $x_{0}\otimes x_{n}^{*}$ denotes the operator $x\mapsto \left\langle x,x_{n}\right\rangle x_{0},$ then for any $n\in\mathbb{N}$ and $a\in\mathbb{C},$ the operator 
\begin{equation}\label{ilil}
	T=J+a (x_{0}\otimes x_{n}^{*})
\end{equation}
is half-centered. 
\end{exa}
		
In fact, the operator~\eqref{ilil} can be seen to be of type~\eqref{wqqw} if we view $\mathcal{H}$ as $L^{2}(\mathbb{N},\mu),$ where $\mu$ is the counting measure. Define $\psi_{n}:\mathbb{N}\rightarrow\mathbb{N}$ by $\psi_{n}(k)=k-1$ if $k\geq 1$ and $\psi_{n}(0)=n$ and let $\xi(k)=a_{k-1}$ if $k\geq 1$ and $\xi(0)=a.$ It is not hard to see that the operator $$f(x)\mapsto \xi(x)f(\psi_{n}(x))$$ coincides with the operator~\eqref{ilil}. Hence, by Proposition~\ref{exempl} the latter is half-centered.

\subsection{The Main Theorem}
The main purpose of this paper is to prove the following result. 
\begin{mthm}\label{mthm1}
	Let $T$ be an injective half-centered operator on $\mathcal H$ such that $\bigvee_{k=0}^{\infty}T^{k}\mathcal{M_{E}}=\mathcal{H}$ and $\dim \left(T\mathcal{H}\right)^{\bot}=1.$
\\
				
Then there are two possibilities (though not mutually exclusive).
\renewcommand{\labelenumi}{\arabic{enumi}}
\begin{enumerate}
\item 
	There is an orthonormal basis $\left\{x_{k}:k\in\mathbb{N}\right\}$ of common eigenvectors for the operators $\left\{T_{k}\right\}_{k\in\mathbb{N}}$ such that with respect to this basis, $T$ is either a weighted shift or there is a weighted shift $J$ such that
	\begin{equation}\label{ebag}
	T=J+a (x_{0}\otimes x_{n}^{*})
	\end{equation}
	for a $n\in\mathbb{N}$ and $a\in\mathbb{C}.$

\item
	There are constants $a,b,c,d\in \mathbb{R},$ not all zero and $k,n\in \mathbb{N}^{+}$ such that 
	\begin{equation}\label{gabe}
	aI+bT^{*k}T^{k} +c T^{*n}T^{n}+ d T^{*k+n}T^{k+n}=0.
	\end{equation}
					
	\end{enumerate}
Moreover, if $\dim\mathcal{M_{E}}\geq 3$ then~\eqref{gabe} holds with $a\neq 0$ and the range of $T$ is closed.
\end{mthm}
			
\begin{rem}
Notice that if $\dim \mathcal{M_{E}}=1$ then $\mathcal{M_{E}}=\E$ and hence $T^{*k}T^{k}\E\subseteq \E$ for all $k\in\mathbb{N}.$ By Proposition~\ref{surject}, $T$ is centered and the condition $\bigvee_{k=0}^{\infty}T^{k}\E=\mathcal{H}$ gives that $T$ is a weighted shift.
\end{rem}
So far, we have not given any concrete example of a half-centered operator where $\dim \mathcal{M_{E}}\geq 3.$ In order to show that this class is not just void, we construct below a half-centered operator having the property that $\mathcal{M_{E}}$ is the whole space.
			
\begin{exa}\label{head}
Let $\mathcal{H}=\ell^{2}$ with standard basis $\left\{e_{k}:k\in\mathbb{N}\right\}$ and let $S$ be the shift operator.
For $ 0< q <1,$ let $A_{q}$ be the operator that in the basis $e_{k}$ can be written as the infinite matrix 
\begin{equation}\label{qq}
	A_{q}=\left[ \begin{matrix}
		0 & 1 & 0 & 0 & .. \\
		1 & 0 & q & 0 & .. \\
		0 & q & 0 & q^{2} & .. \\
		0 & 0 & q^{2} & 0 & .. \\
	.. & .. & .. & .. & .. \\
	\end{matrix} \right].
\end{equation}
Since $ 0< q <1,$ it is straightforward to deduce that $A_{q}$ is a compact self-adjoint operator. Moreover, it is easy to see that
\begin{equation}\label{qed}
	S^{*}A_{q}S=q A_{q} 
\end{equation}
and $\ker A_{q}=\left\lbrace 0\right\rbrace .$
Thus $\ell^{2}$ has an orthonormal basis $\left\{x_{k}:k\in\mathbb{N}\right\}$ consisting of eigenvectors for $A_{q}$ and we can easily deduce that $\left\langle x_{k}, e_{0}\right\rangle \neq 0$ for all $k$ which implies that every eigenspace of $A_{q}$ must be one-dimensional. 
				
Since $A_{q}$ is self-adjoint, there is $r>0$ such that $A_{q}+r I$ is invertible and positive. Now let 
\begin{equation}
T=(A_{q}+r I)^{\frac{1}{2}}S(A_{q}+r I)^{-\frac{1}{2}}.
\end{equation}
Then 
\begin{equation}\label{qe}
T^{n}=(A_{q}+r I)^{\frac{1}{2}}S^{n}(A_{q}+r I)^{-\frac{1}{2}}
\end{equation}
and so by~\eqref{qed}, we see that 
\begin{equation}\label{astttr}
T^{*n}T^{n}=(A_{q}+r I)^{-\frac{1}{2}}(q^{n}A_{q}+r I)(A_{q}+r I)^{-\frac{1}{2}}
\end{equation}
from which it follows that $(T^{*n}T^{n})(T^{*m}T^{m})=(T^{*m}T^{m})(T^{*n}T^{n})$ for $m,n\geq 0$ and hence $T$ is half-centered. Furthermore, if $\lambda_{k}$ is the eigenvalue of the eigenvector $x_{k}$ for $A_{q},$ then $x_{k}$ is clearly an eigenvector for $T^{*n}T^{n},$ with eigenvalue $\frac{q^{n}\lambda_{k}+r}{\lambda_{k}+r}.$ Since the function $\frac{q^{n}x+r}{x+r}$ is one to one on $(-r,\infty),$ we get that $T^{*n}T^{n}$ has only one-dimensional eigenspaces. From the formula~\eqref{qe}, we have $$\E=\ker T^{*}=(A_{q}+r I)^{-\frac{1}{2}}e_{0}$$ giving $\left\langle \E,x_{k}\right\rangle\neq 0$ for all $k.$ If $V$ were a nontrivial closed subspace, invariant under the $T_{k}$'s and orthogonal to $\E,$ then $V$ would have to contain a nonzero eigenvector $x_{m}$ of $A_{q},$ giving $\left\langle \E,x_{m} \right\rangle =0,$ a contradiction. Since the operators $T_{k}$ are all self adjoint, also $\mathcal{M_{E}}^{\bot}$ is invariant with respect to them and so by the last sentence, we must have $\mathcal{M_{E}}^{\bot}=\left\lbrace 0 \right\rbrace $ giving $\mathcal{M_{E}}=\ell^{2}.$
\end{exa}
It can be seen from~\eqref{qe} that the operator defined in Example~\ref{head} satisfies the equation
\begin{equation}
I-(1+q^{-1})T^{*}T+q^{-1} T^{*2}T^{2}=0.
\end{equation}
This is similar to the one that defines the $2$-isometries.
Indeed, the $2$-isometries are a natural occurring example where often $\dim\mathcal{M_{E}}\geq 3$, although the way they usually are constructed makes this a bit cumbersome to check.

		\section{Theory for general injective operators}
			Before we can tackle the main theorem we must first build up some machinery.
			\\
			
			While the theory presented in this section was developed specifically to deal with the half-centered operators, it turned out that it could, with minor extra work, be generalize it to a more general setting. Hence it is presented in this fashion.
			\\
			
			Let us fix some more notation:
			\\
			
			$\mathcal{H}$ is a separable complex Hilbert space with inner product $\left\langle \cdot , \cdot \right\rangle$ and $R$ is a fixed bounded injective linear operator on $\mathcal{H}.$
			Let $\mathcal{B} \left(\mathcal{H}\right)$ be the space of bounded operators on $\mathcal{H}$ and denote $\ker R^{*}$ by $\E.$  Throughout the rest of this text the letter $T$ will be reserved for injective half-centered operators.
			Given a closed subspace $V$ of the Hilbert space $H$ we write $P_{V}$ for the orthogonal projection onto $V.$
			Also, for an operator $B$ and a subspace $V$ of $\mathcal{H}$ we write the restriction of $B$ to $V$ as $B|_{V}$ (or sometimes, to avoid multiple index, we write $B|V$ instead). Notice that if $V$ is an invariant subspace for $B$ then $$\left(B|_{V}\right)^{k}=B^{k}|_{V}$$ for all $k\in\mathbb{N}.$ When we have an algebra of operators $\mathcal{A}\subseteq \mathcal{B}\left(\mathcal{H}\right)$ and a subspace $V$ which is invariant under all operators in $\mathcal{A},$ then $\mathcal{A}|V\subseteq \mathcal{B}\left(V\right)$ is the $C^{*}$-algebra of operators that consists of elements in $\mathcal{A}$ restricted to $V.$
			\\
			
			The main idea of this section is to decompose the subspace $\bigvee_{m=0}^{\infty}R^{m}\mathcal{M_{E}}$ into a direct product $\oplus_{m=0}^{\infty}V_{m}$ of orthogonal subspaces $V_{m}$ with $V_{0}=\mathcal{M_{E}},$ such that $R$ acts on each $V_{m}$ in a "reasonable" predictable way. Moreover, each $V_{m}$ will be an invariant subspace for all the operators $R^{*k}R^{k}.$ We will furthermore show that there is a strong relation between the restriction of $R^{*k}R^{k}$ to different $V_{m}$'s in the sense that there is a natural surjective homomorphism from a sub-algebra of the von Neumann algebra generated by the operators $\left\lbrace R^{*k}R^{k}|V_{m};k\in\mathbb{N} \right\rbrace $ onto the von Neumann algebra generated by $\left\lbrace R^{*k}R^{k}|V_{n};k\in\mathbb{N} \right\rbrace $ when $n\geq m.$
			This construction makes up the technical core of this text, but will take some time to complete.

			\subsection{The $C^{*}$-algebras $\mathbb{M}_{R,n}$ and $\mathbb{M}_{R}^{n}$}
				
			The purpose of this subsection is to introduce two sequences of $C^{*}$-algebras $\mathbb{M}_{R,n}$ and $\mathbb{M}_{R}^{n},$ both indexed over $\mathbb{N}.$ For some basic theory about $C^{*}$-algebras, we recommend~\cite{muphy}.
				\\ 
				
				We also remind the reader of the notation $$R_{k}=R^{*k}R^{k}$$ that will be used for the remainder of the text.
				Note that if $V$ is an invariant subspace for $R,$ then 
				\begin{equation}\label{millian}
					\left(R|_{V}\right)_{k}=\left(R|_{V}\right)^{*k}\left(R|_{V}\right)^{k}=P_{V}R^{*k}R^{k}P_{V}|_{V}=R_{k}|_{V}.
				\end{equation}
				We will for technical reasons often not differentiate between the restriction of an operator $A$ to a subspace $V$ and $P_{V}A P_{V},$ so for example, we write the equality~\eqref{millian} as $\left(R|_{V}\right)_{k}=P_{V}R_{k}P_{V}.$ This is hopefully never a source of confusion. 
				To further simplify notation, we write $$\mathcal{H}_{n}=\overline{R^{n}\mathcal{H}}.$$
				Notice that although we may have $\mathcal{H}_1\neq \mathcal{H},$ this does not in general imply $\mathcal{H}_{n+1}\neq \mathcal{H}_{n}$ for all $n\in\mathbb{N}$.
				\\
				
				Next, we are going to define some of the main objects studied in this section: 
				\\
				
				Let $\mathbb{M}_{R}$ be the von Neumann algebra generated by the operators $R_{k}$ for all $k\in\mathbb{N}$.
				\\
				
				If $\theta_{R}$ is the isometric part of the polar decomposition of $R$ i.e $R=\theta_{R} R_{1}^{\frac{1}{2}},$ let $\mathbb{M}_{R}^{1}$ be the von Neumann algebra generated by the operators

				\begin{center}
					$\theta_{R}^{*}R_{j}\theta_{R}$ for all $j\in\mathbb{N}.$
				\end{center}
				
				If $R$ has a closed range, then $R_{1}$ is invertible, so $\theta_{R}=R R_{1}^{-\frac{1}{2}}$ and thus in this case we have $$\theta_{R}^{*}R_{j}\theta_{R}=R_{1}^{-\frac{1}{2}}R_{j+1}R_{1}^{-\frac{1}{2}}\in\mathbb{M}_{R}.$$ So for closed range $R$ it is easy to see that $\mathbb{M}_{R}^{1}$ is a sub-algebra of $\mathbb{M}_{R}.$ This is also true in general:
				
				\begin{prop}\label{argg}\label{pust}
					The von Neumann algebra $\mathbb{M}_{R}^{1}$ is a sub-algebra of $\mathbb{M}_{R}.$
					Moreover, $\mathbb{M}_{R}^{1}$ is isomorphic to $\mathbb{M}_{R|\mathcal{H}_{1}}.$
					
				\end{prop}
				
				\begin{proof}
					Since $\mathbb{M}_{R}$ is von Neumann algebra, we have by the double commutant theorem 
					$$(\mathbb{M}_{R}')'=\mathbb{M}_{R}$$ where $\mathcal{A}'$ denotes the commutant of the algebra $\mathcal{A}.$ Let $m$ be an element in $\mathbb{M}_{R}'.$
					Since $\theta_{R} R_{1}^{\frac{1}{2}}=R,$ we have $$R_{1}^{\frac{1}{2}}\theta_{R}^{*}R_{j}\theta_{R} R_{1}^{\frac{1}{2}}=R^{*}R_{j}R=R_{j+1}.$$
					thus 
					
					$$R_{1}^{\frac{1}{2}}\theta_{R}^{*}R_{j}\theta_{R} m R_{1}^{\frac{1}{2}}=R_{1}^{\frac{1}{2}}\theta_{R}^{*}R_{j}\theta_{R} R_{1}^{\frac{1}{2}}m=mR_{1}^{\frac{1}{2}}\theta_{R}^{*}R_{j}\theta_{R} R_{1}^{\frac{1}{2}}=R_{1}^{\frac{1}{2}}m\theta_{R}^{*}R_{j}\theta_{R} R_{1}^{\frac{1}{2}}.$$ 
					
					If $R$ is injective then the range of $R_{1}^{\frac{1}{2}}$ is dense in $\mathcal{H},$ this gives $$\theta_{R}^{*}R_{j}\theta_{R} m=m \theta_{R}^{*}R_{j}\theta_{R}$$ for all $m\in \mathbb{M}_{R}'$ so that $$\theta_{R}^{*}R_{j}\theta_{R}\in \mathbb{M}_{R}''=\mathbb{M}_{R}.$$
					For the second claim, note that the map $$B\in\mathcal{B(H)}\mapsto \theta_{R} B \theta_{R}^{*}$$ is a isomorphism $\mathcal{B}\left(\mathcal{H}\right)\rightarrow \mathcal{B}\left(\mathcal{H}_{1}\right)$ such that $$\theta_{R}^{*}R_{j}\theta_{R}\mapsto \theta_{R} \theta_{R}^{*}R_{j}\theta_{R} \theta_{R}^{*}=P_{\mathcal{H}_{1}}R_{j}P_{\mathcal{H}_{1}}=\left(R|_{\mathcal{H}_{1}}\right)_{j}$$ for all $j\in \mathbb{N}.$ Since $\mathbb{M}_{R|\mathcal{H}_{1}}$ is generated by these operators and the map is weakly continuous, the range must be equal to $\mathbb{M}_{R|\mathcal{H}_{1}}.$   
				\end{proof}
				
				By Propositions~\ref{argg}, there is an injective homomorphism 
				$$\mathbb{M}_{R}\leftarrow \mathbb{M}_{R|\mathcal{H}_{1}}$$
				$$\left(R|_{\mathcal{H}_{1}}\right)_{j}=P_{\mathcal{H}_{1}}R_{j}P_{\mathcal{H}_{1}}\mapsto \theta_{R}^{*}R_{j}\theta_{R}.$$
				If we now consider $R|_{\mathcal{H}_{1}}$ instead, we get by the same reasoning that there is an injective homomorphism
				$$\mathbb{M}_{R|\mathcal{H}_{1}}\leftarrow \mathbb{M}_{R|\mathcal{H}_{2}}.$$ So by induction, there is a sequence of injective homomorphisms
				
				\begin{equation}\label{cath}
					\mathbb{M}_{R}\leftarrow \mathbb{M}_{R|\mathcal{H}_{1}}\leftarrow \mathbb{M}_{R|\mathcal{H}_{2}}\leftarrow \mathbb{M}_{R|\mathcal{H}_{3}}\leftarrow \mathbb{M}_{R|\mathcal{H}_{4}}\leftarrow ..
				\end{equation} 
				where the $n$'th arrow is induced by $\theta_{R|\mathcal{H}_{n-1}}:\mathcal{H}_{n-1}\rightarrow \mathcal{H}_{n}.$ 
				Since the maps in~\eqref{cath} are all injective, we can deduce
				
				\begin{prop}\label{basbas}
					If $T$ is half-centered, then $T|_{\mathcal{H}_{n}}$ is also half-centered.
				\end{prop}
				
				We see that the composition $\mathbb{M}_{R}\leftarrow\mathbb{M}_{R|\mathcal{H}_{n}}$ is induced by an isometry 
				$\theta_{R,n}:\mathcal{H}\rightarrow \mathcal{H}_{n}$
				given by the product 
				\begin{equation}\label{quake3}
					\theta_{R,n}=\theta_{R|\mathcal{H}_{n}}\cdot \theta_{R|\mathcal{H}_{n-1}}\cdot...\cdot \theta_{R|\mathcal{H}_{0}}.
				\end{equation}
				We set $\theta_{R,1}=\theta_{R}$
				and $\theta_{R,0}=I.$
				We will identify $\theta_{R,n}$ with the map on $\mathcal{H}$ given by 
				$$x\in\mathcal{H}\mapsto \left(0,\theta_{R,n}x\right)\in \mathcal{H}_{n}^{\bot}\oplus \mathcal{H}_{n}=\mathcal{H}$$
				So that $$\theta_{R,n}\theta_{R,n}^{*}=P_{\mathcal{H}_{n}}.$$
				More generally, $\theta_{R|\mathcal{H}_{n}}$ is interpreted as a partial isometry (that fails to be left-invertible if $\mathcal{H}_{n}\neq \mathcal{H}$) that is zero on $\mathcal{H}_{n}^{\bot}$ and maps $\mathcal{H}_{n}\rightarrow \mathcal{H}_{n+1}.$
				\\
				
				For a half-centered operator $T$ the isometries~\eqref{quake3} can be described as follows.
				
				\begin{prop}\label{key}
					If $T$ is injective and half-centered, and $T^{n}=\theta_{T^{n}}T_{n}^{\frac{1}{2}}$ is the polar decomposition of $T^{n},$ then
					$$\theta_{T,n}=\theta_{T^{n}}$$ 
					Therefor, if $T$ has closed range, then $$\theta_{T,n}=T^{n}T_{n}^{-\frac{1}{2}}.$$
				\end{prop}
				
				The proof will be given after we prove Lemma~\ref{labann}

				\begin{rem}
					An important result in the theory of centered operators is that $\theta_{T}^{n}=\theta_{T^{n}},$ the above proposition can be seen as a generalization of this.
				\end{rem}

				Next we define a class of sub-algebras of $\mathbb{M}_{R}.$
				
				\begin{defn}
					For every $n\in\mathbb{N},$ we define the von Neumann algebra $\mathbb{M}_{R}^{n}$ to be the weakly closed sub-algebra of $\mathbb{M}_{T}$ generated by the operators $\theta_{R,n}^{*}R_{j}\theta_{R,n}.$ By Lemma~\ref{ridcully} below, this algebra can alternatively be defined as the image of $\mathbb{M}_{R|\mathcal{H}_{n}}$ inside $\mathbb{M}_{R}$ using the composition of homomorphisms in~\eqref{cath}.
				\end{defn}
				
				We write down some direct consequences the preceding definitions:
				
				\begin{lem}\label{laban}
					For all $k,n\in\mathbb{N}$
					\begin{equation}
						\theta_{R|\mathcal{H}_{k},n}\theta_{R,k}=\theta_{R,n+k}
					\end{equation}

				\end{lem}
				
				\begin{lem}\label{ridcully}
					The image of $\left(R|_{\mathcal{H}_{n}}\right)_{k}\in \mathbb{M}_{R|\mathcal{H}_{n}}$ in $\mathbb{M}_{R}^{n}$ is $\theta_{R,n}^{*}R_{k}\theta_{R,n}$ and hence $\mathbb{M}_{R}^{n}$ is equal to the image of $\mathbb{M}_{R|\mathcal{H}_{n}}$ in $\mathbb{M}_{R}$ by composition of the homomorphisms in~\eqref{cath}.
				\end{lem}
				
				\begin{proof}
					We have $\left(R|_{\mathcal{H}_{n}}\right)_{k}=P_{\mathcal{H}_{n}}R_{k}P_{\mathcal{H}_{n}}$
					since $\mathcal{H}_{n}$ is an invariant subspace for $R$ and also $P_{\mathcal{H}_{n}}\theta_{R,n}=\theta_{R,n}.$ Therefore the image of $\left(R|_{\mathcal{H}_{n}}\right)_{k}$ in $\mathbb{M}_{R}$ is given by $\theta_{R,n}^{*}R_{k}\theta_{R,n}.$ The second part is obvious as the operators $\left(R|_{\mathcal{H}_{n}}\right)_{k}$ generates $\mathbb{M}_{R|\mathcal{H}_{n}}$ and the homomorphisms in~\eqref{cath} is weakly continuous.
				\end{proof}
				
				\begin{cor}
					If $T$ is half-centered and has closed range, then $\mathbb{M}_{R}^{n}$ is generated by the operators $T_{k+n}T_{n}^{-1}$ (in the weak operator topology).
				\end{cor}
				
				\begin{proof}
					By Proposition~\ref{key} 
					$$\theta_{T,n}=T T_{n}^{-\frac{1}{2}}.$$
					From this we get $$\theta_{T,n}^{*}T_{k}\theta_{T,n}=T_{k+n}T_{n}^{-1}.$$
				\end{proof}
				
				Next we introduce another class of $C^{*}$-algebras associated to $R$ called $\mathbb{M}_{R,n}.$ These will in general be non-unital weakly closed algebras of $\mathcal{B}(\mathcal{H})$ that has $\mathcal{H}_{n}$ as an invariant subspace and 
				$\mathbb{M}_{R,n} \mathcal{H}_{n}^{\bot}=0.$ Moreover, $\mathbb{M}_{R,n}| \mathcal{H}_{n}$ is a von Neumann algebra such that $\mathbb{M}_{R,n}|\mathcal{H}_{n}\cong\mathbb{M}_{R}$ by Proposition~\ref{stön}.
				\\
				
				For every $n\in\mathbb{N},$ take the set of operators $R^{n}\mathbb{M}_{R}R^{*n}=\left\lbrace R^{*n}aR^{n}: a\in\mathbb{M}_{R}\right\rbrace $ and let $\mathbb{M}_{R,n}$ to be the weak closure of this set.
				We let $\mathbb{M}_{R}=\mathbb{M}_{R,0}.$

				\begin{lem}
					$\mathbb{M}_{R,n}$ is a $C^{*}$-algebra.
				\end{lem}
				
				\begin{proof} 
					Additive and adjoint closeness are obvious. 
					If $a,b\in \mathbb{M}_{R},$ then $$R^{n}a R^{*n}R^{n}b R^{n}=R^{n}c R^{*n}$$ with $c=a R^{*n}R^{n}b\in \mathbb{M}_{R}.$ The rest follows now from continuity.
				\end{proof}
				
				Next, we will see that $\theta_{R,n}$ induces a isomorphism between $\mathbb{M}_{R}$ and $\mathbb{M}_{R,n}$ given
				by the mapping $$m\mapsto \theta_{R,n} m \theta_{R,n}^{*}.$$
				To prove this, we first need a technical lemma.
				\begin{lem}\label{labann}
					For every $n\in\mathbb{N},$ there is an operator $r_{n}\in\mathbb{M}_{R}$ such that 
					\begin{equation}\label{thet}
						\theta_{R,n}r_{n}=R^{n}.
					\end{equation}
					and
					\begin{equation}\label{thet2}
						r_{n}^{*}r_{n}=R_{n}.
					\end{equation}
					Moreover, $r_{n}$ has dense range and is given by the formula
					\begin{equation}\label{rform}
						r_{n}=\left(\theta_{R,n-1}^{*}R_{1}\theta_{R,n-1}\right)^{\frac{1}{2}}\cdot \left(\theta_{R,n-2}^{*}R_{1}\theta_{R,n-2}\right)^{\frac{1}{2}}\cdot...\cdot \left(\theta_{R,0}^{*}R_{1}\theta_{R,0}\right)^{\frac{1}{2}}.
					\end{equation}
				\end{lem}
				
				\begin{proof}
					We use induction. For $n=1,$ then $\theta_{R,1}=\theta_{R}$ and so $r_{n}=R_{1}^{\frac{1}{2}}.$ Now assume~\eqref{rform} is true for $n \geq 1,$
					then $$R^{n+1}=R P_{\mathcal{H}_{n}}R^{n}=R P_{\mathcal{H}_{n}}\theta_{R,n}r_{n}.$$
					We have $$R P_{\mathcal{H}_{n}}=\theta_{R|\mathcal{H}_{n}}\left(P_{\mathcal{H}_{n}}R_{1} P_{\mathcal{H}_{n}}\right)^{\frac{1}{2}}.$$ Since $P_{\mathcal{H}_{n}}=\theta_{R,n}\theta_{R,n}^{*}$ and $m\mapsto \theta_{R,n}m\theta_{R,n}^{*}$ is a homomorphism of $C^{*}$-algebras (recall that $\theta_{R,n}$ is an isometry), we have $$\theta_{R|\mathcal{H}_{n}}\left(P_{\mathcal{H}_{n}}R P_{\mathcal{H}_{n}}\right)^{\frac{1}{2}}=\theta_{R|\mathcal{H}_{n}}\theta_{R,n}\left(\theta_{R,n}^{*}R_{1}\theta_{R,n}\right)^{\frac{1}{2}}\theta_{R,n}^{*}$$ $$=\theta_{R,n+1}\left(\theta_{R,n}^{*}R_{1}\theta_{R,n}\right)^{\frac{1}{2}}\theta_{R,n}^{*}.$$ Putting this together, we get $$R^{n+1}=\theta_{R,n+1}\left(\theta_{R,n}^{*}R_{1}\theta_{R,n}\right)^{\frac{1}{2}}r_{n}.$$
					From this~\eqref{thet},~\eqref{thet2} and~\eqref{rform} follow for $n+1.$ Since every operator $\left(\theta_{R,k}^{*}R_{1}\theta_{R,k}\right)^{\frac{1}{2}}$ has dense range, the same is true for their product $r_{n}.$
				\end{proof}
				
				We can now prove Proposition~\ref{key}. Let $t_{n}\in \mathbb{M}_{T}$ be the operator from Lemma~\ref{labann} such that $\theta_{T,n}t_{n}=T^{n}.$ As $$t_{n}=\prod_{k=0}^{n-1}\left(\theta_{T,k}^{*}T_{1}\theta_{T,k}\right)^{\frac{1}{2}}$$ and every $\left(\theta_{T,k}^{*}T_{1}\theta_{T,k}\right)^{\frac{1}{2}}\in\mathbb{M}_{T},$  $t_{n}$ is a product of positive operators that commute with each other, hence it is also positive.
				Now since $$\left(t_{n}\theta_{T,n}^{*}\right)\left(\theta_{T,n}t_{n}\right)=t_{n}^{2}=T_{n}$$ we must have $t_{n}=T_{n}^{\frac{1}{2}},$ by the uniqueness of the square root of a positive operator. 
				So $\theta_{T,n}T_{n}^{\frac{1}{2}}=\theta_{T,n}t_{n}=T^{n}$ and as $T_{n}^{\frac{1}{2}}$ has dense range, we have $\theta_{T,n}=\theta_{T^{n}}.$

				\begin{prop}\label{stön}
					For every $n\in\mathbb{N},$ the homeomorphism $m\mapsto \theta_{R,n}m \theta_{R,n}^{*}$ is an isomorphism $$\mathbb{M}_{R}\rightarrow \mathbb{M}_{R,n}.$$ 
				\end{prop}
				
				\begin{proof}
					
					For any $c\in\mathbb{M}_{R},$ we have $r_{n}c r_{n}^{*}\in\mathbb{M}_{R}$ and this operator is mapped to $R^{n}c R^{*n}$ by Lemma~\ref{labann}. The homomorphism preserves weak closure (since it is induced by an isometry) so $$\mathbb{M}_{R,n}\subseteq\theta_{R,n}\mathbb{M}_{R}\theta_{R,n}^{*}.$$ To prove the reverse inclusion, take any $m \in \mathbb{M}_{R}.$ Since $r_{n}r_{n}^{*}$ has dense range, there is a sequence of self-adjoint $y_{k}\in\mathbb{M}_{R}$ such that $$y_{k}r_{n}r_{n}^{*},r_{n}r_{n}^{*}y_{k}\rightarrow I$$ strongly in $\mathcal{H}$ as $k\rightarrow \infty$ (this follows from a basic application of the general spectral theorem). Now take the product $$r_{n}^{*} y_{k} m y_{k} r_{n}\in\mathbb{M}_{R}$$ for every $k\in\mathbb{N}.$
					Then we have $$R^{n}\left(r_{n}^{*} y_{k} m y_{k} r_{n}\right)R^{*n}\in\mathbb{M}_{R,n}$$ for all $k\in\mathbb{N}.$ But since $R^{*n}=r_{k}^{*}\theta_{R,n}^{*},$ we get
					$$R^{n}\left(r_{n}^{*} y_{k} m y_{k} r_{n}\right)R^{*n}=\theta_{R,n}\left(r_{n}r_{n}^{*}y_{k}\right)m\left(y_{k}r_{n}r_{n}^{*}\right)\theta_{R,n}^{*}\rightarrow \theta_{R,n} m \theta_{R,n}^{*}$$ strongly. So $ \theta_{R,n} m \theta_{R,n}^{*}\in\mathbb{M}_{R,n}$ and thus $$\theta_{R,n}\mathbb{M}_{R}\theta_{R,n}^{*}=\mathbb{M}_{R,n}.$$
				\end{proof}

				A consequence can be directly drawn from Proposition~\ref{stön}.
				
				\begin{cor}
					For every $n\in\mathbb{N},$ the $C^{*}$-algebra $\mathbb{M}_{R|\mathcal{H}_{n}}$ is a sub-algebra of $\mathbb{M}_{R,n}.$
				\end{cor}	
				\begin{rem}
					Similar to what was mentioned in the introduction to this subsection, we mostly view $\mathbb{M}_{R,n}$ and $\mathbb{M}_{R|\mathcal{H}_{n}}$ as non-unital weakly closed $C^{*}$-algebras in $\mathcal{B(H)}$ rather than unital $C^{*}$-algebras in $\mathcal{B(H}_{n})$ that perhaps would seem more natural. This is because in the upcoming sections, the main job of these algebras are to act on $\mathcal{H}$ and therefore it would be cumbersome if we first always have to project down  $\mathcal{H}_{n}$ before they can be applied. 
				\end{rem}
				
				\subsection{A Subspace Decomposition}
					
					Here we will first decompose the Hilbert space $\mathcal{H}$ into $\mathcal{H_{E}}\oplus \mathcal{H_{E}}^{\bot},$ where $\mathcal{H_{E}}$ is the smallest closed subspace containing $\E$ that is invariant with respect to both $R$ and $\mathbb{M}_{R}.$ We then show that there is a further decomposition of $\mathcal{H_{E}}$ into orthogonal subspaces $$\mathcal{H_{E}}=\oplus_{k=0}^{\infty}V_{k}$$
					with $V_{0}=\mathcal{M_{E}}$ such that all the $V_{k}$'s are invariant subspaces for the algebra $\mathbb{M}_{R}.$ The important point of this construction emerges in the next subsection where we show that $\mathbb{M}_{R}|V_{k}$ 
					and $\mathbb{M}_{R}|V_{0}$ are related in a certain way.
					\\
					
					From now on $R$ will, as well as being injective, also be subject to the condition
					$\E\neq 0$ (recall $\E=\mathcal{H}_{1}^{\bot}=\left(R\mathcal{H}\right)^{\bot}=\ker R^{*}$)
					
					Also recall from the introduction that $\mathcal{M_{E}}$ was defined as the linear closure of $$m y,m\in \mathbb{M}_{R},y\in \mathcal{E}.$$	
					
					For notational purposes, we sometimes abbreviate this as $\mathbb{M}_{R}\mathcal{E}$ and this notation will be used from now on in general, when we have a $C^{*}$-algebra or a set of operators acting on some subspace. All subspaces here will be considered as closed, unless explicitly stated. So, for example, given subspaces $V,X\subseteq \mathcal{H}$ the subspace $R X +V$ will denote the closure of $$\left\lbrace R x+v;x\in X,v\in V \right\rbrace$$
					\\
					
					We remark that as $\E\subseteq\mathcal{M_{E}}$ the subspace $\mathcal{M_{E}}$ is invariant for $P_{\mathcal{H}_{1}}$ and hence also invariant under the operators $$\left(R|_{\mathcal{H}_{1}}\right)_{k}=P_{\mathcal{H}_{1}}R_{k}P_{\mathcal{H}_{1}}.$$

					\begin{lem}\label{3}
						For all $m\in \mathbb{N}$ and $n\leq m$ $$\mathbb{M}_{R,m}R^{m}\mathcal{E}=R^{n}\mathbb{M}_{R,m-n}R^{m-n}\mathcal{E}=R^{m}\mathcal{M_{E}}.$$
					\end{lem}
					
					\begin{proof}
						We prove the equality
						\begin{equation}\label{lyud}
							\mathbb{M}_{R,m}R^{m}\E=R^{m}\mathcal{M_{E}}.
						\end{equation} 
						The rest of the Lemma then follows from $$R^{n}\mathbb{M}_{R,m-n}R^{m-n}\mathcal{E}=R^{n}\mathbb{M}_{R,m-n}R^{m-n}\mathcal{E}=R^{n}(R^{m-n}\mathcal{M_{E}})=R^{m}\mathcal{M_{E}}.$$
						Since $\mathbb{M}_{R,m}\E$ and $R^{m}\mathcal{M_{E}}$ are both closed subspaces of $\mathcal{H}_{m}$ and $P_{\mathcal{H}_{m}}=\theta_{R,m}\theta_{R,m}^{*},$ we can prove~\eqref{lyud} by proving that 
						\begin{equation}\label{ska}
							\theta_{R,m}^{*}\mathbb{M}_{R,m}R^{m}\E=\mathcal{M_{E}}=\theta_{R,m}^{*}R^{m}\mathcal{M_{E}}.
						\end{equation}
						
						But since by Proposition~\ref{stön} $\mathbb{M}_{R,m}=\theta_{R,m}^{*}\mathbb{M}_{R}\theta_{R,m}^{*}$ and $\theta_{R,m}^{*}R^{m}=r_{m}$ where $r_{m}\in\mathbb{M}_{R}$ is as in Lemma~\ref{labann}, we have $\theta_{R,m}^{*}R^{m}\mathcal{M_{E}}=r_{m}\mathcal{M_{E}}$
						and $$\theta_{R,m}^{*}\mathbb{M}_{R,m}R^{m}\E=\mathbb{M}_{R}r_{m}\E.$$
						From this it is now obvious that they are both subspaces of $\mathcal{M_{E}}.$
						Since the range of $r_{m}$ is dense in $\mathcal{H}$ and $\mathcal{M_{E}}$ is an invariant subspace for both $r_{m}$ and $r_{m}^{*},$ we must have $r_{m}\mathcal{M_{E}}=\mathcal{M_{E}}.$ So the second equality in~\eqref{ska} is proven. To prove the first, recall $r_{m}^{*}r_{m}=R_{m}$ and take a sequence $a_{k}\in\mathbb{M}_{R}$ such that $a_{k}R_{m}\rightarrow I$ strongly. Then the sequence $a_{k}r_{m}^{*}\in\mathbb{M}_{R}$ is such that $(a_{k}r_{m}^{*})r_{m}\rightarrow I$ strongly. From this we see that $\E\subseteq \mathbb{M}_{R}r_{m}\E$ and since the space in question is also invariant under $\mathbb{M}_{R},$ it must be equal to $\mathcal{M_{E}}.$
						
					\end{proof}
					
					\begin{lem}\label{saknar}
						Let $\mathcal{E}_{n}=\mathcal{H}_{n}\ominus \mathcal{H}_{n+1}$ be the kernel of $R^{*}$ restricted to $\mathcal{H}_{n}.$ For all $n\in\mathbb{N},$ we have $\mathcal{E}_{n}\subseteq  R^{n}\mathcal{M_{E}}.$
					\end{lem}
					
					\begin{proof}
						Since $\mathcal{E}_{n}\bot \mathcal{H}_{n+1},$ we have for any $x\in\mathcal{H}$ and $e\in \mathcal{E}_{n}$ that $$0=\left\langle e,R^{n+1}x\right\rangle=\left\langle  R^{*n}e,R x\right\rangle.$$ From this, we see $R^{*n}\mathcal{E}_{n}\subseteq \left(R\mathcal{H}\right)^{\bot}= \mathcal{E}$ so that $R^{n}\mathbb{M}_{R}R^{*n}\mathcal{E}_{n}\subseteq R^{n}\mathcal{M_{E}}.$ Since $\mathbb{M}_{R,n}$ is the weak closure of the set $R^{n}\mathbb{M}_{R}R^{*n},$ we get $\mathbb{M}_{R,n}\mathcal{E}_{n}\subseteq R^{n}\mathcal{M_{E}}.$  Now, we have $P_{\mathcal{H}_{n}}\in\mathbb{M}_{R,n}$ and so $P_{\mathcal{H}_{n}}\mathcal{E}_{n}=\mathcal{E}_{n}\subseteq R^{n}\mathcal{M_{E}}.$
					\end{proof}
					
					\begin{defn}
						For $n\in \mathbb{N},$ let $X_{n}=\bigvee_{ j=0}^{n}R^{j}\mathcal{M_{E}}.$ When $n\geq 1,$ define $$V_{n}=X_{n}\ominus X_{n-1}$$
						and when $n=0,$ let $V_{0}=\mathcal{M_{E}}.$
					\end{defn}
					
					\begin{lem}\label{nolabel}
						Each $V_{n}$ and $X_{n}$ is $\mathbb{M}_{R}$ invariant.
					\end{lem}
					
					\begin{proof} We use induction on $n$.
						The lemma is true by construction for $V_{0}=\mathcal{M_{E}}.$ Since the $R_{j}$'s are self-adjoint, we only have to show that $X_{n}$ is $R_{j}$-invariant for all $j\in \mathbb{N}.$ Assume now that the lemma is true for $0\leq m\leq n-1.$ The vectors of the form $$x_{n}=m_{n}+x_{n-1}$$ with $m_{n}\in R^{n}\mathcal{M_{E}}$ and $x_{n-1}\in X_{n-1}$ are dense in $X_{n}$ and for these: 
						
						\begin{equation}\label{mabel}
							R_{j}x_{n}=R_{j}m_{n}+R_{j}x_{n-1}=P_{\mathcal{H}_{n}}R_{j}P_{\mathcal{H}_{n}}m_{n}+\left( I-P_{\mathcal{H}_{n}}\right) R_{j}  m_{n}+R_{j}x_{n-1}.
						\end{equation}
						We have $P_{\mathcal{H}_{n}}R_{j}P_{\mathcal{H}_{n}}m_{n}\in R^{n}\mathcal{M_{E}}$ by Lemma~\ref{3}, since $P_{\mathcal{H}_{n}}R_{j}P_{\mathcal{H}_{n}}\in \mathbb{M}_{R,n}.$
						\\
						
						As $\mathcal{E}_{j}=\mathcal{H}_{j}\ominus \mathcal{H}_{j+1}$ we have $$\bigvee_{j=0}^{n-1}\mathcal{E}_{j}=\mathcal{H}_{n}^{\bot}$$
						and therefore $I-P_{\mathcal{H}_{n}}=P_{\mathcal{H}_{n}^{\bot}}$ projects down to the space generated by $\mathcal{E}_{j}$ for $0\leq j\leq n-1$ and as $\mathcal{E}_{j}\subset \mathcal{M_{E}}$ by Lemma~\ref{saknar},
						we have $\left( I-P_{\mathcal{H}_{n}}\right)R_{j}   m_{n}\in X_{n-1}.$
						\\
						
						By induction, we have $R_{j}x_{n-1} \in X_{n-1}$ and hence all the vectors on the right hand side of~\eqref{mabel} are in $X_{n}.$
					\end{proof}
					
					\begin{cor}\label{fckkth}\label{fuckkkk}
						For $x\in X_{n-1}^{\bot}$ we have $R_{j}x=\left(R|_{\mathcal{H}_{n}}\right)_{j}x$ and
						$X_{n-1}$ is an invariant subspace for $ \mathbb{M}_{R|\mathcal{H}_{n}}$
					\end{cor}
					
					\begin{proof}
						As $\ker R^{*n}=\mathcal{H}_{n}^{\bot}=\bigvee_{k=0}^{n-1}\E_{k},$ we have by Lemma~\ref{saknar} that $\ker R^{*n} \subseteq \bigvee_{ j=0}^{n-1}R^{j}\mathcal{M_{E}} = X_{n-1}$ and hence $X_{n-1}^{\bot}\subseteq \mathcal{H}_{n}.$ By Lemma~\ref{nolabel}, $X_{n-1}^{\bot}$ is invariant with respect to $R_{j}.$ Hence for all $x\in X_{n-1}^{\bot},$ we have $P_{\mathcal{H}_{n}}x=x$ and so $$R_{j}x=P_{\mathcal{H}_{n}}R_{j}P_{\mathcal{H}_{n}}x=\left(R|_{\mathcal{H}_{n}}\right)_{j}x.$$
						
						To see the second claim, we observe that since $X_{n-1}$ is an invariant subspace for $P_{\mathcal{H}_{n}},$ it is an invariant subspace for all 
						$$P_{\mathcal{H}_{n}}R_{k}P_{\mathcal{H}_{n}}=\left(R|_{\mathcal{H}_{n}}\right)_{k}$$
						and hence for $\mathbb{M}_{R|\mathcal{H}_{n}}.$
					\end{proof}
					
					\begin{prop}\label{isisis}
						For all $m\geq n$ we have $P_{V_{m}} R^{m-n} V_{n}=V_{m}.$
					\end{prop}
					
					\begin{proof}
						
						Since  $V_{m}=X_{m}\ominus X_{m-1}$ and $X_{m}=X_{m-1}+R^{m}\mathcal{M_{E}}$ we have $V_{m}=P_{V_{m}}R^{m}\mathcal{M_{E}}.$ So for every $v_{m}\in V_{m}$ and $\epsilon > 0$ there is a $v_{0}\in \mathcal{M_{E}}$ such that $\left\|v_{m}-P_{V_{m}}R^{m}v_{0}\right\|< \epsilon.$ Then $P_{V_{m}}R^{m}v_{0} \neq 0$ will imply $P_{V_{n}}R^{n}v_{0}\neq 0$ since $P_{V_{n}}R^{n}v_{0}=0$ would imply $R^{n}v_{0}\in X_{n-1}$ and so $R^{m}v_{0}\in X_{m-1}\bot V_{m}.$ Since $P_{V_{k}}R^{k}v_{0}=R^{k}v_{0}-x_{k-1}$ with $x_{k-1}\in X_{k-1}$ we then have $$P_{V_{m}}R^{m-n}P_{V_{n}}R^{n}v_{0}=P_{V_{m}}R^{m-n}\left(R^{n}v_{0}-x_{n-1}\right)=P_{V_{m}}R^{m}v_{0}$$
						and hence $$\left\|v_{m}-P_{V_{m}}R^{m-n}P_{V_{n}}R^{n}v_{0}\right\|<\epsilon.$$
						
					\end{proof}
					
					Proposition~\ref{isisis} directly implies:
					
					\begin{cor}\label{e11}\label{ockey}
						For all $m\geq n,$ we have $\dim V_{m}\leq \dim V_{n}.$ Especially, if
						$V_{N}=\left\{0\right\}$ for some $N\in \mathbb{N}$ then $V_{n}=\left\{0\right\}$ for all $n\geq N.$
					\end{cor}
					
					So if $\dim\mathcal{M_{E}}=V_{0}$ is finite, then $\dim V_{n}$ must be finite for all $n\in\mathbb{N}.$ The injectivity of $R$ now gives:
					
					\begin{prop}\label{e1}
						If $V_{0}=\mathcal{M_{E}}$ is finite dimensional, then $V_{k}\neq\left\{0\right\}$ for all $k\in\mathbb{N}$
					\end{prop}
					
					\begin{proof}
						If there would be a $K\in\mathbb{N}$ such that $V_{K}=\left\{0\right\}$ then Corollary~\ref{ockey} would give $V_{k}=\left\{0\right\}$ for all $k\geq K.$ By Theorem~\ref{space1}, $\mathcal{H}_{\E}$ must be finite dimensional and mapped by $R$ again into $\mathcal{H}_{\E}$ with a nontrivial cokernel, but then there must be a nonzero vector $v\in\mathcal{H}_{1}$ such that $R v=0,$ a contradiction.
					\end{proof}

					\begin{thm}\label{space1} 
						For an injective operator $R$ on a Hilbert space $\mathcal{H},$ the subspace $\mathcal{H}_{\E}=\bigvee_{ j=0}^{\infty}R^{j}\mathcal{M_{E}}$ is $R$ and $\mathbb{M}_{R}$ invariant. Moreover $$\mathcal{H}_{\E}=\oplus_{k=0}^{\infty}V_{k}.$$
						
						and $$R V_{k}\subseteq V_{k+1}\oplus(X_{k}\ominus R X_{k-1}).$$
					\end{thm}
					\begin{proof}
						The decomposition $\mathcal{H}_{\E}=\oplus_{k=0}^{\infty}V_{k}$ follows by definition.
						Since  $$\mathcal{H}_{\E}=\oplus_{k=0}^{\infty}V_{k} =\bigvee_{k=0}^{\infty}R^{k}\mathcal{M_{E}}$$ and $\bigvee_{k=0}^{\infty}R^{k}\mathcal{M_{E}}$ invariant under $R,$ the first claim is proved. 
						
						To prove the last claim, note that since $R X_{k}=\bigvee_{j=0}^{k}R^{j+1}\mathcal{M_{E}},$
						we have $R V_{k}\bot V_{k+2+m},$ for all $m\geq 0.$
						Also, since $V_{k}=X_{k}\ominus X_{k-1}$ and by Lemma~\ref{nolabel} $X_{k-1}$ is invariant under $R_{1}$ $$\left\langle R V_{k},R X_{k-1}\right\rangle=\left\langle V_{k}, R_{1}X_{k-1}\right\rangle=0.$$
						Thus $R V_{k}\bot R^{j}\mathcal{M_{E}}$ for $1\leq j \leq k$ and $R V_{k}\bot\oplus_{m=k+2}^{\infty}V_{m}.$ For $v_{k} \in V_{k},$ we have $R v_{k}=P_{X_{k}^{\bot}}R v_{k}+P_{X_{k}}R v_{k}.$
						Now $P_{X_{k}^{\bot}}Rv_{k}\in V_{k+1}$ and we have $$\left\langle P_{X_{k}}R v_{k}, R^{j}V_{0}\right\rangle=\left\langle R v_{k}, R^{j}V_{0}\right\rangle=0$$ for $1 \leq j\leq k$ so that $$P_{X_{k}}R v_{k}\in X_{k} \ominus R X_{k-1}.$$
					\end{proof}
					
					\begin{cor}\label{fuhåkan}
						If $T$ is injective and half-centered, then $T|\mathcal{H}_{\E}$ is also injective and half-centered.
					\end{cor}
					
					By to Theorem~\ref{1hemma}, if $v\in V_{m}$ and $R v\bot X_{m}\ominus R X_{m-1},$ then we get $Rv\in V_{m+1}.$ In the context here, this is not a particular useful characterization of those $v\in V_{m}$ that end up in $V_{m+1}$ when applying $R.$ As we will see in the next proposition, it turns out that while we may not have $X_{m+1}\ominus R X_{m}\subseteq\mathcal{M_{E}},$ none of the vectors in $R V_{k}\ominus V_{k+1}$ can be orthogonal to $\mathcal{M_{E}}$. 
					
					\begin{prop}\label{jups}
						If $v_{m}\in V_{m}$ and $R v_{m}\bot \mathcal{M_{E}},$ then $Rv_{m}\in V_{m+1}.$
					\end{prop}
					
					\begin{proof}
						From the proof of Theorem~\ref{space1} we know $Rv_{m}\bot\bigvee_{j=1}^{m}R^{j}\mathcal{M_{E}},$ so also $Rv_{m}\bot\mathcal{M_{E}}$ would imply that $Rv_{m}\in X_{m+1}$ and $Rv_{m}\bot \bigvee_{j=0}^{m}R^{j}\mathcal{M_{E}}=X_{m}$ i.e $Rv_{m}\in V_{m+1}.$
					\end{proof}
					
					\begin{cor}\label{jugs}
						If $v_{m}\in V_{m}$ and $\theta_{R} v_{m}\bot \mathcal{M_{E}},$ then $\theta_{R} v_{m}\in V_{m+1}.$
					\end{cor}
					
					\begin{proof}
						Since $R_{1}$ has dense range, there is a sequence $R_{1}x_{k}\in V_{m}$ such that $R_{1}x_{k}\rightarrow v_{m}.$ Then $\theta_{R} R_{1}x_{k}=R x_{k} \rightarrow \theta_{R} v_{m}\bot V_{0}$ and the same arguments as in Proposition~\ref{jups} shows that $\theta_{R} v_{m}\in V_{m+1}.$ 
					\end{proof}

					\begin{lem}\label{fuio}
						For $m\geq n$ the projection $P_{V_{m}}$ commutes with the operators in $\mathbb{M}_{R|\mathcal{H}_{n}}.$
					\end{lem}
					
					\begin{proof}
						Since every $V_{m}$ is invariant under $\mathbb{M}_{R},$ we have that $P_{V_{m}}$ must commute with all $R_{j}.$  By Corollary~\ref{fckkth}
						we have $$\left(R|_{\mathcal{H}_{n}}\right)_{j}P_{V_{m}}=R_{j}P_{V_{m}}=P_{V_{m}}R_{j}=P_{V_{m}}\left(R|_{\mathcal{H}_{n}}\right)_{j}.$$
					\end{proof}

					A consequence of Corollary~\ref{fckkth} is that
					
					\begin{prop}\label{fukth}
						We have
						$$R_{j}|V_{m}=\left(R|_{\mathcal{H}_{n}}\right)_{j}|V_{m}$$ when $m\geq n,$ so that $$\mathbb{M}_{R|\mathcal{H}_{n}}|V_{m}=\mathbb{M}_{R}|V_{m}.$$
						
						Moreover, for $m\geq n$ and all $j\in\mathbb{N}$ we have
						\begin{equation}\label{backto}
							\theta_{R|\mathcal{H}_{n},j}^{*}\left(R|_{\mathcal{H}_{n}}\right)_{i}\theta_{R|\mathcal{H}_{n},j}|V_{m}=
							\theta_{R,j}^{*}R_{i}\theta_{R,j}|V_{m}
						\end{equation}
						so that
						\begin{equation}\label{future}
							\mathbb{M}_{R|\mathcal{H}_{n}}^{j}|V_{m}=\mathbb{M}_{R}^{j}|V_{m}.
						\end{equation}
					\end{prop}
					Where the slighly complicated expression
					$$\theta_{R|\mathcal{H}_{n},j}^{*}\left(R|_{\mathcal{H}_{n}}\right)_{i}\theta_{R|\mathcal{H}_{n},j}\in\mathbb{M}_{R|\mathcal{H}_{n}}^{j}$$ is the image of $\left(R|_{\mathcal{H}_{n+j}}\right)_{i}$ under the homomorphism $\mathbb{M}_{R|\mathcal{H}_{n+j}}\rightarrow \mathbb{M}_{R|\mathcal{H}_{n}}.$

					\begin{proof}
						We will prove~\eqref{backto} by induction on $j.$ The other claim~\eqref{future} then follows from the fact that the operators in~\eqref{backto} generate $\mathbb{M}_{R|\mathcal{H}_{n}}^{j}|V_{m}.$ From Lemma~\ref{nolabel} and Corollary~\ref{fckkth} it follows that $$R_{i}|V_{m}=\left(R|_{\mathcal{H}_{n}}\right)_{i}|V_{m}$$ so the claim is true for $j=0.$
						\\
						
						Now, assume it is true for $j-1\geq k\geq 0.$ By Lemma~\ref{labann} there is a $r_{n,j}\in \mathbb{M}_{R|\mathcal{H}_{n}}$ such that $\theta_{R|\mathcal{H}_{n},j}r_{n,j}=R^{j}|\mathcal{H}_{n}$ and hence
						
						$$r_{n,j}^{*}\left(\theta_{R|\mathcal{H}_{n},j}^{*}\left(R|_{\mathcal{H}_{n}}\right)_{i}\theta_{R|\mathcal{H}_{n},j}\right)r_{n,j}=$$
						$$\left(r_{n,j}^{*}\theta_{R|\mathcal{H}_{n},j}^{*}\right)\left(R|_{\mathcal{H}_{n}}\right)_{i}\left(\theta_{R|\mathcal{H}_{n},j}r_{n,j}\right)=\left(R|_{\mathcal{H}_{n}}\right)_{i+j}.$$
						
						If we also take $r_{j}\in\mathbb{M}_{R}$ such that $\theta_{R,j}r_{j}=R^{j}$ then $$r_{j}^{*}\left(\theta_{R,j}^{*}R_{i}\theta_{R,j}\right)r_{j}=R_{i+j}.$$ We want to prove that $$r_{n,j}|V_{m}=r_{j}|V_{m}.$$ But this follows from the induction hypothesis, as the formula for
						$r_{n,j}$ is given by $$r_{n,j}=\left(\theta_{R|\mathcal{H}_{n},j-1}^{*}\left(R|_{\mathcal{H}_{n}}\right)_{1}\theta_{R|\mathcal{H}_{n},j-1}\right)^{\frac{1}{2}}\cdot ... \cdot \left(\theta_{R|\mathcal{H}_{n},0}^{*}\left(R|_{\mathcal{H}_{n}}\right)_{1}\theta_{R|\mathcal{H}_{n},0}\right)^{\frac{1}{2}}$$ and since we assumed that~\eqref{backto} was true for $j-1\geq k\geq 0,$ we get $r_{n,j}|V_{m}=r_{j}|V_{m}.$
						\\
						
						We can now calculate
						$$r_{j}^{*}\left(\theta_{R,j}^{*}R_{i}\theta_{R,j}\right)r_{j}|V_{m}=R_{i+j}|V_{m}=\left(R|_{\mathcal{H}_{n}}\right)_{i+j}|V_{m}=$$
						$$=r_{n,j}^{*}\left(\theta_{R|\mathcal{H}_{n},j}^{*}\left(R|_{\mathcal{H}_{n}}\right)_{i}\theta_{R|\mathcal{H}_{n},j}\right)r_{n,j}|V_{m}=r_{j}^{*}\left(\theta_{R|\mathcal{H}_{n},j}^{*}\left(R|_{\mathcal{H}_{n}}\right)_{i}\theta_{R|\mathcal{H}_{n},j}\right)r_{j}|V_{m}$$ and since $r_{j}$ has dense range, we must have $$\theta_{R,j}^{*}R_{i}\theta_{R,j}|V_{m}=\theta_{R|\mathcal{H}_{n},j}^{*}\left(R|_{\mathcal{H}_{n}}\right)_{i}\theta_{R|\mathcal{H}_{n},j}|V_{m}.$$
						
						hence~\eqref{backto} is also true for $j.$ 
						
					\end{proof}
					
					\subsection{A connection between $\mathbb{M}_{R}|V_{n}$ and $\mathbb{M}_{R}|V_{m}$}
						In the previous subsection we found a decomposition of $\mathcal{H_{E}}=\bigvee_{k=0}^{\infty}R^{k}\mathcal{M_{E}}$ into subspaces $V_{n}$ which are invariant with respect to the algebra $\mathbb{M}_{R}.$
						Here we show that there is a natural way to connect the different restrictions $\mathbb{M}_{R}|V_{n}$ and $\mathbb{M}_{R}|V_{m}.$ This will be essential in the proof of the main theorem.
						To explain what this connection is, we need some results that are proven below. Theorem~\ref{normspos} shows that for all $n,m\in\mathbb{N}$
						such that $m\geq n,$ there is a surjective homomorphism $$\Gamma_{m,n}:\mathbb{M}_{R}^{m-n}|V_{n}\rightarrow \mathbb{M}_{R}|V_{m}$$ which, in particular, maps
						\begin{equation}\label{omgds1}
							\theta_{R,m-n}^{*}R_{k}\theta_{R,m-n}|V_{n}\mapsto R_{k}|V_{m}
						\end{equation}
						and more generally
						\begin{equation}\label{omgds}
							\theta_{R,m-n+j}^{*}R_{k}\theta_{R,m-n+j}|V_{n}\mapsto \theta_{R,j}R_{k}\theta_{R,j}|V_{m}
						\end{equation}
						for all $j\geq 0.$ There is also the inclusion homomorphism $\mathbb{M}_{R}^{m-n}|V_{n}\hookrightarrow \mathbb{M}_{R}|V_{n},$ so we have the following diagram:
						\begin{equation}\label{hurray}
							\begin{tikzcd}
								\mathbb{M}_{R}|V_{m}\\
								\mathbb{M}_{R}^{m-n}|V_{n} \arrow[hookrightarrow]{r}{}[swap]{}
								\arrow[rightarrow]{u}{}[swap]{\Gamma_{m,n}}
								&\mathbb{M}_{R}|V_{n}
							\end{tikzcd}.
						\end{equation}
						From~\eqref{omgds}, the homomorphisms $\Gamma_{m,n}$ also "preserves" the sub-algebras $\mathbb{M}_{R}^{j}$ in the sense that the restriction of $\Gamma_{m,n}$ to $\mathbb{M}_{R}^{m-n+j}|V_{n}$ is a surjective homomorphism
						$$\mathbb{M}_{R}^{m-n+j}|V_{n}\rightarrow \mathbb{M}_{R}^{j}|V_{m}.$$

						Another feature of the homomorphisms $\Gamma_{m,n}$ is that they factors through $m\geq i\geq n$ so that the following diagram commutes
						
						\begin{equation}\label{hurray2}
							\begin{tikzcd}
								&\mathbb{M}_{R}|V_{m} \arrow[leftarrow]{d}{\Gamma_{m,i}}[swap]{}\\
								\mathbb{M}_{R}^{m-n}|V_{n} \arrow[rightarrow]{r}{}[swap]{\Gamma_{i,n}}
								\arrow[rightarrow]{ur}{}[swap]{\Gamma_{m,n}}
								&\mathbb{M}_{R}^{m-i}|V_{i}
							\end{tikzcd}.
						\end{equation}

						We start with a particular example.
						\begin{exa}\label{exor}
							Let $T$ be a left invertible weighted shift on $\ell^{2}$ (thus $T$ is centered) and let $\left\{x_{k}:k\in\mathbb{N}\right\}$ denote the standard basis of $\ell^{2},$ so that $Tx_{k}=a_{k} x_{k+1},$ with $a_{k}\in\mathbb{C}$ and $a_{k}\neq 0.$ Then the kernel of $T^{*}$ is $\left\langle x_{0}\right\rangle ,$ the subspace generated by $x_{0}.$ Since there is $\lambda_{k}\in\mathbb{R}$ such that $$T_{k}x_{0}=\lambda_{k}x_{0}$$ for all $k\in\mathbb{N},$ we have $\mathcal{M}_{\E}=\left\langle x_{0}\right\rangle .$ From this we can deduce $$V_{k}=\left\langle x_{k}\right\rangle .$$ Moreover, it is also easy to see that $\theta_{T}=T T_{1}^{-\frac{1}{2}}$ is an isometric shift on the basis $\left\{x_{k}:k\in\mathbb{N}\right\}$ and $$\theta_{T,k}=\theta_{T^{k}}=T^{k}T_{k}^{-\frac{1}{2}}$$ (for a proof of this, use Proposition~\ref{key}).  Then~\eqref{hurray} and~\eqref{hurray2} can be seen as a generalization of the fact that for any $m,n,j\in\mathbb{N}$ with $m\geq n,$ we have $$T_{j}x_{m}=\frac{\lambda_{m+j}}{\lambda_{m}}x_{m}$$ and $$\theta_{T^{m-n}}^{*}T_{j}\theta_{T^{m-n}}x_{n}=T_{m-n}^{-1}T_{j+ m-n}x_{n}=\frac{\lambda_{n +(m-n)+j}}{\lambda_{n}}\frac{\lambda_{n}}{\lambda_{m}}x_{n}=\frac{\lambda_{m+j}}{\lambda_{m}}x_{n}.$$
						\end{exa}
						
						It is good to keep Example~\ref{exor} in mind, since there all the components defined in this section ($\mathbb{M}_{R},V_{k},\Gamma_{m,n}$ etc) become very simple.
						
						\begin{lem}\label{galax2}
							For $m\geq n,$ the operator $\theta_{R|\mathcal{H}_{n},m-n}$ is a bijective isometry $$R^{n}\mathcal{M_{E}}\rightarrow R^{m}\mathcal{M_{E}}$$
							that induces an isomorphism
							$$\Theta_{m,n}:\mathbb{M}_{R,n}|R^{n}\mathcal{M_{E}}\rightarrow \mathbb{M}_{R,m}|R^{m}\mathcal{M_{E}}$$
							given by 
							\begin{equation}\label{tabs}
								\Theta_{m,n}:b \mapsto \theta_{R|\mathcal{H}_{n},m-n}b \theta_{R|\mathcal{H}_{n},m-n}^{*}
							\end{equation} 
							for  $b\in \mathbb{M}_{R,n}.$
							Moreover, if $m\geq i\geq n,$ then 
							\begin{equation}\label{tabs84}
								\Theta_{m,i}\Theta_{i,n}=\Theta_{m,n}
							\end{equation}
						\end{lem}
						
						\begin{proof}
							We have $$R^{m}=\theta_{R|\mathcal{H}_{n},k}r_{n,m-n}R^{n}$$ were $r_{n,m-n}$ is the same as in Proposition~\ref{fukth}. Since $r_{n,m-n}\in\mathbb{M}_{R,n}$ has dense range in $\mathcal{H}_{n},$ we get $$R^{m}\mathcal{M_{E}}=\theta_{R|\mathcal{H}_{n},m-n}r_{n,m-n}R^{n}\mathcal{M_{E}}=\theta_{R|\mathcal{H}_{n},m-n}R^{n}\mathcal{M_{E}}.$$
							Now, since $$\theta_{R|\mathcal{H}_{n},m-n}\mathbb{M}_{R,n}\theta_{R|\mathcal{H}_{n},m-n}^{*}=\mathbb{M}_{R,m}$$ as $\theta_{R|\mathcal{H}_{n},m-n}\theta_{R,n}=\theta_{R,m},$ it is not hard to see that~\eqref{tabs} defines an isomorphism $$\Theta_{m,n}:\mathbb{M}_{R,n}|R^{n}\mathcal{M_{E}}\rightarrow \mathbb{M}_{R,m}|R^{m}\mathcal{M_{E}}.$$ The property~\eqref{tabs84} follows from $$\theta_{R|\mathcal{H}_{n},m-n}=\theta_{R|\mathcal{H}_{i},m-i}\theta_{R|\mathcal{H}_{n},i-n}$$ so that $$\theta_{R|\mathcal{H}_{n},m-n}\mathbb{M}_{R,n}\theta_{R|\mathcal{H}_{n},m-n}^{*}=\theta_{R|\mathcal{H}_{i},m-i}\theta_{R|\mathcal{H}_{n},i-n}\mathbb{M}_{R,n}\theta_{R|\mathcal{H}_{n},i-n}^{*}\theta_{R|\mathcal{H}_{i},m-i}^{*}=$$  $$\theta_{R|\mathcal{H}_{n},m-i}\mathbb{M}_{R,i}\theta_{R|\mathcal{H}_{n},m-i}^{*}=\mathbb{M}_{R,m}.$$
						\end{proof}

						\begin{lem}~\label{galax}
							For every $n\in\mathbb{N}$ there is a surjective homomorphism $$\Phi_{n}:\mathbb{M}_{R|\mathcal{H}_{n}}|R^{n}\mathcal{M_{E}}\rightarrow \mathbb{M}_{R}|V_{n}$$
							given by $$\left(R|_{\mathcal{H}_{n}}\right)_{j}|R^{n}\mathcal{M_{E}}\mapsto R_{j}|V_{n}.$$
							Furthermore, $\Phi_{n}$ restricts to a surjective homomorphism $$\mathbb{M}_{R|\mathcal{H}_{n}}^{k}|R^{n}\mathcal{M_{E}}\rightarrow \mathbb{M}_{R}^{k}|V_{n}$$ that maps $$\theta_{R|\mathcal{H}_{n},k}\left(R|_{\mathcal{H}_{n}}\right)_{j}\theta_{R|\mathcal{H}_{n},k}|R^{n}\mathcal{M_{E}}\mapsto \theta_{R,k}^{*}R_{j}\theta_{R,k}|V_{n}$$
							for all $k\geq 0.$
						\end{lem}
						
						\begin{proof}
							Since $\mathbb{M}_{R|\mathcal{H}_{n}}$ is a sub-algebra of $\mathbb{M}_{R,n}$ and by Lemma~\ref{3} $\mathbb{M}_{R,n}R^{n}\mathcal{M_{E}}\subseteq R^{n}\mathcal{M_{E}}$ the restriction map $$\eta_{n}:\mathbb{M}_{R|\mathcal{H}_{n}}\rightarrow \mathbb{M}_{R|\mathcal{H}_{n}}|R^{n}\mathcal{M_{E}}$$
							$$b\mapsto b|R^{n}\mathcal{M_{E}}$$
							is a homomorphism.
							By Proposition~\ref{fukth}, the map $$\xi_{n}:\mathbb{M}_{R|\mathcal{H}_{n}}\rightarrow \mathbb{M}_{R|\mathcal{H}_{n}}|V_{n}=\mathbb{M}_{R}|V_{n}$$ that sends $\left(R|_{\mathcal{H}_{n}}\right)_{j}$ to $R_{j}|V_{n}$ is a homomorphism. Now $P_{V_{n}}R^{n}\mathcal{M_{E}}=V_{n},$ so if we take any $m_{n}\in \ker \eta_{n},$ then $$m_{n}P_{V_{n}}R^{n}\mathcal{M_{E}}=P_{V_{n}}m_{n}R^{n}\mathcal{M_{E}}=0$$ by Lemma~\ref{fuio}. Hence the map $\Phi_{n}:\xi_{n}(b)\mapsto \eta_{n}(b), b\in \mathbb{M}_{R|\mathcal{H}_{n}}$ is a well-defined surjective homomorphism from $\mathbb{M}_{R|\mathcal{H}_{n}}|R^{n}\mathcal{M_{E}}$ to $\mathbb{M}_{R}|V_{n}.$
							The second claim follows from Proposition~\ref{fukth}. 
							
						\end{proof}
						
						\begin{thm}\label{normspos}
							There are surjective homomorphisms $$\Gamma_{m,n}:\mathbb{M}_{R}^{m-n}|V_{n}\rightarrow\mathbb{M}_{R}|V_{m}$$ that maps 
							\begin{equation}\label{omgds2}
								\theta_{R,m-n+j}^{*}R_{k}\theta_{R,m-n+j}|V_{n}\mapsto \theta_{R,j}^{*}R_{k}\theta_{R,j}|V_{m}
							\end{equation}
							for all $j\geq 0.$
							Furthermore, for every $n\leq i\leq m$ then $\Gamma_{i,n}$ restricts to a homomorphism $$\Gamma_{i,n}:\mathbb{M}_{R}^{m-n}|V_{n}\rightarrow\mathbb{M}_{R}^{m-i}|V_{i}$$ such that $$\Gamma_{m,i}\Gamma_{i,n}=\Gamma_{m,n}.$$
						\end{thm}
						
						\begin{proof}
							Combining Lemma~\ref{galax} and Lemma~\ref{galax2}, we get a diagram
							
							\begin{equation}\label{hurray3}
								\begin{tikzcd}
									\mathbb{M}_{R|\mathcal{H}_{n}}^{m-n}|R^{n}\mathcal{M_{E}} \arrow[rightarrow]{r}{}[swap]{\Theta_{m,n}}\arrow[rightarrow]{d}{}[swap]{\Phi_{n}}&\mathbb{M}_{R|\mathcal{H}_{m}}|R^{m}\mathcal{M_{E}}\arrow[rightarrow]{d}{}[swap]{\Phi_{m}}\\
									\mathbb{M}_{R}^{m-n}|V_{n} &\mathbb{M}_{R}|V_{m}
								\end{tikzcd}.
							\end{equation}
							and we want to prove that there is an unique $$\Gamma_{m,n}:\mathbb{M}_{R}^{m-n}|V_{n}\rightarrow \mathbb{M}_{R}|V_{m}$$ that makes this diagram commutative.
							Making the composition $$\Phi_{m}\Theta_{m,n}:\mathbb{M}_{R|\mathcal{H}_{n}}^{m-n}|R^{n}\mathcal{M_{E}} \rightarrow \mathbb{M}_{R}|V_{m}$$ we need to prove $\ker \Phi_{n}\subseteq\ker \Phi_{m}\Theta_{m,n} ,$ because then we can define $\Gamma_{m,n}$ as the map sending $\Phi_{n}(b)$ to $\Phi_{m}\Theta_{m,n}(b)$ for $b\in \mathbb{M}^{m-n}_{R|\mathcal{H}_{n}}|R^{n}\mathcal{M_{E}}.$ 
							\\
							
							If we take $b\in\mathbb{M}_{R|\mathcal{H}_{n}}^{m-n}|R^{n}\mathcal{M_{E}}$ such that $\Phi_{n}\left(b\right)=0,$ then as $P_{V_{n}}R^{n}\mathcal{M_{E}}=V_{n} $ and $b$ commutes with $P_{V_{n}}$ by Lemma~\ref{fuio}, this implies $$\Phi_{n}\left(b\right)V_{n}=b P_{V_{n}}R^{n}\mathcal{M_{E}}=P_{V_{n}}b R^{n}\mathcal{M_{E}}=0,$$
							so 
							\begin{equation}\label{miin}
								b R^{n}\mathcal{M_{E}}\subseteq X_{n-1}.
							\end{equation}
							We have also that $\Phi_{m}\Theta_{m,n}\left(b\right)=0$ implies 
							\begin{equation}\label{miin2}
								\theta_{R|\mathcal{H}_{n},m-n}bR^{n}\mathcal{M_{E}}\subseteq X_{m-1}.
							\end{equation}
							We want to prove that~\eqref{miin} implies~\eqref{miin2}.
							To show this, we prove the more general statement that $$\theta_{R|\mathcal{H}_{n},m-n}X_{n-1}\subseteq X_{m-1}.$$ The partial isometry $\theta_{R|\mathcal{H}_{n},m-n}$ has a kernel equal to $\mathcal{H}_{n}^{\bot},$ so 
							$$\theta_{R|\mathcal{H}_{n},m-n}X_{n-1}=\theta_{R|\mathcal{H}_{n},m-n}\left(X_{n-1}\ominus\mathcal{H}_{n}^{\bot}\right).$$ We know that there is a $r_{n,m-n}\in\mathbb{M}_{R|\mathcal{H}_{n}}$ with dense range in $\mathcal{H}_{n}$ such that $\theta_{R|\mathcal{H}_{n},m-n}r_{n,m-n}=R^{m-n}|_{\mathcal{H}_{n}}$ and by Corollary~\ref{fuckkkk}
							$$r_{n,m-n}\left(X_{n-1}\ominus\mathcal{H}_{n}^{\bot}\right)=X_{n-1}\ominus\mathcal{H}_{n}^{\bot}.$$
							From this we can deduce 
							$$\theta_{R|\mathcal{H}_{n},m-n}\left(X_{n-1}\ominus\mathcal{H}_{n}^{\bot}\right)=R^{m-n}\left(X_{n-1}\ominus\mathcal{H}_{n}^{\bot}\right)\subseteq X_{m-1}.$$ This gives the existence of $\Gamma_{m,n}.$
							Surjectivity follows from $\Phi_{m}\Theta_{m,n}=\Gamma_{m,n}\Phi_{n}$ and $\Phi_{m}\Theta_{m,n}$ is surjective. Uniqueness follows from the surjectivity of $\Phi_{n}.$ 
							\\
							
							The property~\eqref{omgds} follows from applying the commutative diagram to $$\theta_{R|\mathcal{H}_{n},m-n+j}\left(R|_{\mathcal{H}_{n}}\right)_{i}\theta_{R|\mathcal{H}_{n},m-n+j}\in \mathbb{M}_{R|\mathcal{H}_{n}}^{m-n}|R^{n}\mathcal{M_{E}}$$
							for $j\geq 0.$ The property~\eqref{hurray2} follows, as remarked, from~\eqref{omgds}.
						\end{proof}
						
						While the next result is not used in the proof of the main theorem, it showcases nicely, if $\mathcal{H}_{\E}=\mathcal{H},$ how the behavior of the operators $R_{k}$ restricted to $\mathcal{M_{E}}$ can be connected to their behavior on the whole space.
						
						\begin{cor}\label{halfline}
							If $\mathcal{H}_{\E}=\mathcal{H}$ and $R_{j}R_{k}x=R_{k}R_{j}x$ for all $x\in \mathcal{M_{E}}$ and $k,j\in\mathbb{N}.$
							Then $R$ is half-centered.
						\end{cor}

						With the help of Theorem~\ref{normspos} we can now express the spectrum of $\theta_{T,k}^{*}T_{j}\theta_{T,k}|V_{n}$ via the spectrum of $\mathbb{M}_{T}|\mathcal{M_{E}}.$ 
						
						\begin{prop}\label{loveu}
							Let $T$ be half-centered and if $\gamma$ is a point of the spectrum of $\mathbb{M}_{T}$ restricted $V_{n},$ then there is a point $\lambda $ in the spectrum of $\mathbb{M}_{T}$ restricted to $\mathcal{M_{E}}$ such that $$\gamma \left(\theta_{T,k}^{*}T_{j}\theta_{T,k}\right)=\lambda\left(\theta_{T,k+n}^{*}T_{j}\theta_{T,k+n}\right)$$ for all $j,k \in \mathbb{N}.$
						\end{prop}
						
						Note that for every point $\gamma\in\sigma(\mathbb{M}_{T})$ and all $j,k\in\mathbb{N}$ $$\gamma\left(\theta_{T,k}^{*}T_{j}\theta_{T,k}\right)\gamma\left(T_{k}\right)=\gamma\left(T_{k}^{\frac{1}{2}}\theta_{T,k}^{*}T_{j}\theta_{T,k}T_{k}^{\frac{1}{2}}\right)=\gamma\left(T_{k+j}\right).$$ So, if $\gamma\left(T_{k}\right)\neq 0,$ then $$\gamma\left(\theta_{T,k}^{*}T_{j}\theta_{T,k}\right)=\frac{\gamma\left(T_{k+j}\right)}{\gamma\left(T_{k}\right)}.$$

						

					\section{Fundamentals for Half-Centered operators}
						
						Here we present some initial results that hold for all injective half-centered operators with $\dim\E=1$. 
						Much of the work in this section will aim towards showing that the operator $T_{k}$ has a simple form when restricted to $\mathcal{M_{E}}.$ We will see that there are real parameters $\tau_{k},\beta_{k}$ and a self adjoint operator $A\in\mathcal{B}(\mathcal{M_{E}})$ which is independent of $k,$ such that $T_{k}|_{\mathcal{M_{E}}}$ is given by the formula $$T_{k}|_{\mathcal{M_{E}}}=\tau_{k}I+\beta_{k}A$$ where $I$ is the identity on $\mathcal{M_{E}}.$
						This implies that there are $a,b,c\in \mathbb{R},$ not all zero, and $k,m\in\mathbb{N}^{+}$ such that 
						\begin{equation}\label{joddla}
							aI+b T_{k}+c T_{m}|_{\mathcal{M_{E}}}=0|_{\mathcal{M_{E}}}
						\end{equation}
						which can be seen as a weaker form of the main theorem and indeed if $\mathcal{M_{E}}=\mathcal{H},$ then~\eqref{joddla} directly implies it. However, we can not conclude from~\eqref{joddla} that the same identity must hold for the whole space (and generally it will not).
						The step from linear dependence in $\mathcal{M_{E}}$ to linear dependence in $\mathcal{H}$ is the main obstacle here and much of the theory in section $2$ was introduced as a way to deal with this.
						\\
						
						Since the subspace $\E$ is now one dimensional, we take $\E$ to mean a unit length vector that spans the space. To keep the notations simpler, we also write $P$ instead of $P_{\mathcal{H}_{1}}.$
						\\ 
						We recall the earlier result (Proposition~\ref{basbas}):
						\begin{center}
							\textit{If $T$ is half-centered then so is $T |_{\mathcal{H}_{1}}.$}\label{ioio}
						\end{center}
						
						This implies that $P T_{k}P T_{j}P=P T_{j}P T_{k}P$ for all $j,k\in \mathbb{N}.$ As $P_{\E}=I-P,$ we can deduce  $$PT_{k}P_{\mathcal{E}}T_{j}P=PT_{k}T_{j}P-P T_{k}P T_{j}P=$$ $$PT_{j}T_{k}P-P T_{j}P T_{k}P=PT_{j}P_{\mathcal{E}}T_{k}P$$
						so that 
						\begin{equation}\label{mayi}
							PT_{k}P_{\mathcal{E}}T_{j}P=PT_{j}P_{\mathcal{E}}T_{k}P.
						\end{equation}
						
						This equation leads to the following.
						
						\begin{prop}\label{1hemma}
							For every $x\in \mathcal{H}_{1}$ and $u\in\mathcal{M_{E}}$
							\begin{equation}\label{ggg}
								\left\langle x, T_{m}\mathcal{E}\right\rangle\left(T_{k}-\left\langle T_{k}\mathcal{E}, \mathcal{E}\right\rangle I\right)u=\left\langle x, T_{k}\mathcal{E}\right\rangle\left(T_{m}-\left\langle T_{m}\mathcal{E}, \mathcal{E}\right\rangle I\right)u.
							\end{equation}
						\end{prop}
						\begin{proof}
							First we prove 
							\begin{equation}\label{hemma}
								\left\langle T y ,T_{m}\mathcal{E}\right\rangle\left(T_{k}-\left\langle T_{k}\mathcal{E}, \mathcal{E}\right\rangle I\right)\mathcal{E}=\left\langle T y,T_{k}\mathcal{E}\right\rangle\left(T_{m}-\left\langle T_{m}\mathcal{E}, \mathcal{E}\right\rangle I\right)\mathcal{E}
							\end{equation}
							for each $y\in \mathcal{H}.$
							Since $$PT_{m}\mathcal{E}=\left(T_{m}-\left\langle T_{m}\mathcal{E}, \mathcal{E}\right\rangle I\right)\mathcal{E},$$ we have
							$$PT_{m}P_{\mathcal{E}}T_{k}PT=\left(T_{m}-\left\langle T_{m}\mathcal{E}, \mathcal{E}\right\rangle I\right)P_{\mathcal{E}}T_{k}T.$$ By~\eqref{mayi}, this is the same as $$PT_{k}P_{\mathcal{E}}T_{m}PT=\left(T_{k}-\left\langle T_{k}\mathcal{E}, \mathcal{E}\right\rangle I\right)P_{\mathcal{E}}T_{m}T.$$
							So we have 
							$$\left(T_{m}-\left\langle T_{m}\mathcal{E}, \mathcal{E}\right\rangle I\right)P_{\mathcal{E}}T_{k} T y=\left(T_{k}-\left\langle T_{k}\mathcal{E}, \mathcal{E}\right\rangle I\right)P_{\mathcal{E}}T_{m} T y$$ for all $y\in\mathcal{H}.$
							The equation~\eqref{hemma} now follows from $$P_{\mathcal{E}}T_{m}T y=\left\langle T y, T_{m}\mathcal{E}\right\rangle \mathcal{E}.$$
							Since $\mathbb{M}_{T}$ is commutative, we have $$T_{n}\left\langle T y, T_{m}\mathcal{E}\right\rangle\left(T_{k}-\left\langle T_{k}\mathcal{E}, \mathcal{E}\right\rangle I\right)\mathcal{E}=\left\langle T y, T_{m}\mathcal{E}\right\rangle\left(T_{k}-\left\langle T_{k}\mathcal{E}, \mathcal{E}\right\rangle I\right)T_{n}\mathcal{E}$$ for every $T_{n}.$
							
							The only thing left to prove now is that holds for all $x\in \mathcal{H}_{1},$ but this follows from continuity.
						\end{proof}
						
						The following statement must be known, but since we could not find an exact reference for it, we include the proof for the sake of completeness. 
						\begin{lem}\label{itswrong}
							Let $\mathcal{A}$ be a commutative $C^{*}$-algebra of operators on a Hilbert space $\mathcal{K}$ with a cyclic vector $x\in \mathcal{K}.$
							Then given $a_{1},a_{2}\in \mathcal{A}$ and a point $\lambda$ in the spectrum of $\mathcal{A}$ there is a sequence of vectors $x_{l}\in \mathcal{K}$ such that $$a_{i}x_{l}-\lambda\left(a_{i}\right)x_{l}\rightarrow 0$$ as $l\rightarrow \infty$ for $i=1,2$ and $$\frac{\left\langle a_{i} x_{l},x\right\rangle}{\left\langle x_{l},x\right\rangle}\rightarrow \lambda\left(a_{i}\right)$$ as $l\rightarrow \infty.$
							
						\end{lem}
						
						\begin{proof}
							For simplicity, we write $\hat{a}$ for the Gelfand transform of $a\in\mathcal{A}.$
							As $x$ is a cyclic vector for $\mathcal{A},$ there is an isometric representation $u:\mathcal{K}\rightarrow L^{2}(X,\mu_{x}),$ where $X$ is the Gelfand spectrum of $\mathcal{A}$ and $\mu_{x}$ is the Borel measure on $X$ induced by the positive linear functional on $C(X)$ given by
							$$\hat{a}\mapsto \left\langle a x,x \right\rangle .$$
							Let $\beta_{\epsilon}\left[ \hat{a}_{i}(\lambda)\right] $ denote the open ball in $\mathbb{C}$ centered on $\hat{a}_{i}(\lambda)$ and with radius $\epsilon.$ Now define $$W_{\epsilon}=\hat{a}_{1}^{-1}(\beta_{\epsilon}\left[ \hat{a}_{1}(\lambda)\right] )\cap \hat{a}_{2}^{-1}(\beta_{\epsilon}\left[ \hat{a}_{2}(\lambda)\right] )$$ i.e the set in $X$ such that both $\hat{a}_{1}$ and $\hat{a}_{2}$ has distance less that $\epsilon$ from their value at $\lambda.$ Since $\hat{a}_{1}$ and $\hat{a}_{2}$ are both continuous, $W_{\epsilon}$ is an open set and thus there is a non-constant positive continuous function $g_{\epsilon},$ that is zero on $W_{\epsilon}^{c}.$ Since $\mu_{x}$ is finite and has support all of $X$ (due to the fact that $x$ is cyclic), we can further assume that $\int_{X}\left| g_{\epsilon}(z)\right|^{2}d \mu_{x}(z)=1 $ and as $g_{\epsilon}$ is positive, we have $0<\int_{X} g_{\epsilon}(z)d \mu_{x}(z)<\infty.$

							Now we see that  $$\int_{X}\left|\hat{a}_{i}(\lambda) g_{\epsilon}(z)-\hat{a}_{i}(z)g_{\epsilon}(z)\right|^{2}d \mu_{x}(z)=$$
							$$\int_{W_{\epsilon}}\left|(\hat{a}_{i}(\lambda) -\hat{a}_{i}(z))\right|^{2}\left| g_{\epsilon}(z)\right|^{2}d \mu_{x}(z)< \epsilon^{2}$$ for $1\leq i\leq 2$ and thus $\hat{a}_{i}g_{\epsilon}- \hat{a}_{i}(\lambda) g_{\epsilon}\rightarrow 0$ in $L^{2}(X,\mu_{x})$ as $\epsilon \rightarrow 0.$ Moreover $$\left| \frac{\int_{X} \hat{a}_{i}(z)g_{\epsilon}(z)d \mu_{x}(z)}{\int_{X} g_{\epsilon}(z)d \mu_{x}(z)}-\hat{a}_{i}(\lambda)\right|= \left| \frac{\int_{X} \hat{a}_{i}(z)g_{\epsilon}(z)-\hat{a}_{i}(\lambda)g_{\epsilon}(z)d \mu_{x}(z)}{\int_{X} g_{\epsilon}(z)d \mu_{x}(z)}\right| \leq $$
							$$\frac{\int_{W_{\epsilon}} \left| \hat{a}_{i}(z)-\hat{a}_{i}(\lambda)\right| g_{\epsilon}(z) d \mu_{x}(z)}{\int_{W_{\epsilon}} g_{\epsilon}(z)d \mu_{x}(z)} < \epsilon \cdot \frac{\int_{W_{\epsilon}} g_{\epsilon}(z)d \mu_{x}(z)}{\int_{W_{\epsilon}} g_{\epsilon}(z)d \mu_{x}(z)}=\epsilon$$ for $1\leq i\leq 2.$ Taking $x_{l} = u^{-1} g_{\frac{1}{l}},$ we obtain the statement.
						\end{proof}
						
						\begin{cor}\label{opopo}
							Given two points $\lambda,\mu$ of the spectrum of $\mathbb{M}_{T}$ restricted to $\mathcal{M_{E}}$ and $m_{1},m_{2}\in\mathbb{N},$ then there are two sequences of unit vectors $x_{l},y_{l}\in\mathcal{M_{E}}$ such that
							\begin{center} 
								$\frac{\left\langle T_{m_{i}} x_{l},\mathcal{E}\right\rangle}{\left\langle x_{l},\mathcal{E}\right\rangle}\rightarrow \lambda\left(T_{m_{i}}\right)$ and
								$\frac{\left\langle T_{m_{i}} y_{l},\mathcal{E}\right\rangle}{\left\langle y_{l},\mathcal{E}\right\rangle}\rightarrow \mu\left(T_{m_{i}}\right)$
							\end{center}
							as $l\rightarrow \infty$ for $ i=1, 2.$
						\end{cor}
						
						Now, let $(\lambda,\mu),$ $m_{1},m_{2}\in\mathbb{N}$ and $x_{l},y_{l}\in\mathcal{M_{E}}$ be as in corollary~\ref{opopo}.
						Consider the new sequence
						$$v_{l}=\frac{x_{l}}{\left\langle x_{l},\mathcal{E} \right\rangle} - \frac{y_{l}}{\left\langle y_{l},\mathcal{E} \right\rangle}. $$ 
						Then $v_{l}\bot \mathcal{E}$ for all $l\in\mathbb{N}$ so that $v_{l}\in \mathcal{H}_{1}.$ 
						Moreover, for $i=1,2$
						
						$$\left\langle v_{l}, T_{m_{i}}\mathcal{E}\right\rangle\rightarrow\lambda\left(T_{m_{i}}\right)-\mu\left(T_{m_{i}}\right)$$ as $l\rightarrow \infty.$
						\\
						
						If we apply Proposition~\ref{1hemma} with the sequence $v_{l}$ in the place of $x$ and $k=m_{1},m=m_{2}$ and let $l\rightarrow \infty,$ then we get for every $u\in \mathcal{M_{E}}$
						
						\begin{equation}\label{maineq}
							\left(\lambda\left(T_{m}\right)-\mu\left(T_{m}\right) \right)\left(T_{k}-\left\langle T_{k}\mathcal{E}, \mathcal{E}\right\rangle I\right)u=\left(\lambda\left(T_{k}\right)-\mu\left(T_{k}\right) \right)\left(T_{m}-\left\langle T_{m}\mathcal{E}, \mathcal{E}\right\rangle I\right)u.
						\end{equation}
						
						We can draw some conclusions from this formula.
						
						\begin{prop}\label{nicenice}
							Let $\lambda,\mu \in \sigma(\mathbb{M}_{T}|\mathcal{M_{E}})$ and $\lambda\neq \mu.$ Then $\lambda\left(T_{m}\right)=\mu\left(T_{m}\right)$ for some $m\in\mathbb{N}$ if and only if $$T_{m}\mathcal{E}=\left\langle T_{m}\mathcal{E},\mathcal{E}\right\rangle \mathcal{E}$$ 
							i.e $\mathcal{E}$ is an eigenvector for $T_{m}.$
						\end{prop}
						
						\begin{proof}
							If $k$ is such that $\lambda\left(T_{k}\right)\neq \mu\left(T_{k}\right)$ and $m$ such that $\lambda\left(T_{m}\right)= \mu\left(T_{m}\right).$ Then the left-hand side of~\eqref{maineq} is zero and therefore so is the right-hand side, but since $\lambda\left(T_{k}\right)\neq \mu\left(T_{k}\right).$ This means that 
							$$\left(T_{m}-\left\langle T_{m}\mathcal{E}, \mathcal{E}\right\rangle I\right)\mathcal{E}=0.$$
							The other direction is trivial.
						\end{proof}

						If $\dim \mathcal{M_{E}}\geq 2$ then there must be at least two different point in the spectrum of $\mathbb{M}_{T}$ restricted to $\mathcal{M_{E}},$ this makes it possible to do the following definition,
						
						\begin{defn}
							Let $\dim\mathcal{M_{E}}\geq 2$ and let $(\lambda,\mu)$ be two different points in $ \sigma (\mathbb{M}_{T}|\mathcal{M_{E}}).$ For every $k\in\mathbb{N},$ let 
							
							\begin{equation}\label{bee}
								\beta_{k}:=\lambda\left(T_{k}\right)-\mu\left(T_{k}\right).
							\end{equation}
							
						\end{defn}
						
						\begin{rem}
							Clearly $\beta_{0}=0.$ We note also that if $(\lambda',\mu')$ is another cuple of points in $\sigma (\mathbb{M}_{T}|\mathcal{M_{E}})$ then by Lemma~\ref{hemma1} below we have  $\lambda\left(T_{k}\right)-\mu\left(T_{k}\right)=c(\lambda'\left(T_{k}\right)-\mu'\left(T_{k}\right))$ for a nonzero constant $c\in\mathbb{R}$ and every $k\in \mathbb{N},$ so the sequence $\left\lbrace \beta_{k} \right\rbrace  $ is defined up to a multiplicative constant by a couple of different points in the spectrum $\sigma (\mathbb{M}_{T}|\mathcal{M_{E}}).$
						\end{rem}

						\begin{defn}
							We let
							\begin{equation}\label{tau}
								\tau_{k}:=\left\langle T_{k}\mathcal{E}, \mathcal{E}\right\rangle
							\end{equation} 
							for all $k\in\mathbb{N}.$ 
						\end{defn}

						\begin{lem}\label{hemma1}
							If $\lambda$ is in the spectrum of $\mathbb{M}_{T}|\mathcal{M_{E}}$ then $$\lambda\left(T_{k}\right)=\tau_{k}+A_{\lambda} \beta_{k}$$ for some constant $A_{\lambda}\in \mathbb{R}$ only depending on the $\lambda.$
						\end{lem}
						
						\begin{proof}
							With our new notations~\eqref{maineq} can now be rewritten as
							
							\begin{equation}\label{main3q}
								\beta_{k}\left(T_{m}-\tau_{m}\right)u=\beta_{m}\left(T_{k}-\tau_{k}\right)u.
							\end{equation}
							
							By Lemma~\ref{itswrong} we can find a sequence $\left\lbrace x_{j}\right\rbrace \in\mathcal{M_{E}}$ such $\frac{\left\langle T_{k}x_{j}, \mathcal{E}\right\rangle}{\left\langle x_{j}, \mathcal{E}\right\rangle} \rightarrow \lambda\left(T_{k}\right)$ and $\frac{\left\langle T_{m}x_{j}, \mathcal{E}\right\rangle}{\left\langle x_{j}, \mathcal{E}\right\rangle} \rightarrow \lambda\left(T_{m}\right).$
							Substituting $v$ with $\frac{x_{j}}{\left\langle x_{j}, \mathcal{E}\right\rangle}$ in~\eqref{main3q}, then taking the scalar product with $\mathcal{E}$ on both sides and letting $j\rightarrow \infty,$ we get
							\begin{equation}\label{mainn}
								\beta_{k}\left(\lambda\left(T_{m}\right)-\tau_{m}\right)=\beta_{m}\left(\lambda\left(T_{k}\right)-\tau_{k}\right)
							\end{equation}
							If $\dim \mathcal{M_{E}}\geq 2$ then there must be at least one $m\in\mathbb{N}$ such that $\beta_{m}\neq 0$ and if we take
							$$A_{\lambda}=\frac{\left(\lambda\left(T_{m}\right)-\tau_{m}\right)}{\beta_{m}}$$
							then we see from~\eqref{mainn} that $A_{\lambda}$ is independent of the chosen $k\in\mathbb{N}$ as long as $\beta_{k}\neq 0.$
							So we have $$\lambda\left(T_{k}\right)=\tau_{k}+\frac{\lambda\left(T_{k}\right)-\tau_{k}}{\beta_{k}}\beta_{k}=\tau_{k}+A_{\lambda} \beta_{k}$$ when $\beta_{k}\neq 0$ and when $\beta_{j}=0$ we have from Proposition~\ref{nicenice} that $$\lambda\left(T_{j}\right)=\tau_{j}=\tau_{j}+A_{\lambda} \beta_{j}$$ so that the fomula is valid in this case also.
						\end{proof}
						
						
						
						The results of this subsection can be summarized as follows:
						
						\begin{thm}\label{bt1}
							If $T\in \mathcal{B}\left(\mathcal{H}\right)$ is half-centered and injective with $\dim \left(T\mathcal{H}\right)^{\bot}=1,$ then there are self adjoint operators $A,C \in \mathcal{B}(\mathcal{M_{E}}),$ such that for every $k\in\mathbb{N}$
							\begin{equation}\label{raw}
								T_{k}|_{\mathcal{M_{E}}}=\tau_{k}I+\beta_{k}A.
							\end{equation}
							\begin{equation}\label{wwraw}
								PT_{k}P|_{\mathcal{M_{E}}}=\tau_{k}P+\beta_{k}C.
							\end{equation}
							where $C=P A P.$
						\end{thm}
						
						While $T$ is assumed to be injective, we can not rule out the possibility that $0\notin \spec \mathbb{M}_{T},$ in fact we can not even rule out $0\notin \spec \mathbb{M}_{T}|_{\mathcal{M_{E}}}.$ In the end of section $5$, we will see that if $\bigvee_{k=0}^{\infty}R^{k}\mathcal{M_{E}}=\mathcal{H},$ then actually $0\notin \spec \mathbb{M}_{T}|_{\mathcal{M_{E}}},$ but in general this may not be the case. However, the property $0\in \spec \mathbb{M}_{T}|_{\mathcal{M_{E}}}$ does give quite strong implications regarding the structure of $T$ and we must take these into account in the next section when we add the condition $\bigvee_{k=0}^{\infty}R^{k}\mathcal{M_{E}}=\mathcal{H},$ even though we end up showing the non-existence of such points.

						\begin{lem}\label{colio}
							If $\gamma(T_{k})=0$ for some $\gamma\in\sigma(\mathbb{M}_{T})$ and $k\in\mathbb{N},$ then $\gamma\left(T_{k+j}\right)=0$ for all $j\in\mathbb{N}.$
						\end{lem}
						
						\begin{proof}
							We have $$0=\gamma\left(T_{k}\right)\gamma\left(\theta_{T,k}^{*}T_{j}\theta_{T,k}\right)=\gamma\left((T_{k}^{\frac{1}{2}}\theta_{T,k}^{*})T_{j}(\theta_{T,k}T_{k}^{\frac{1}{2}})\right)=\gamma\left(T_{k+j}\right).$$
						\end{proof}
						
						\begin{prop}\label{queen}
							If $0\in\sigma(T_{k}|\mathcal{M_{E}})$ for some $k\in\mathbb{N},$ then 
							$$\beta_{k+j}=\frac{\tau_{j+k}}{\tau_{k}}\beta_{k}$$ and
							$$\theta_{T,k}^{*}T_{j}\theta_{T,k}|_{\mathcal{M_{E}}}=\frac{\tau_{j+k}}{\tau_{k}}I|_{\mathcal{M_{E}}}$$
							for all $j\in\mathbb{N}.$
						\end{prop}
						
						\begin{proof}
							It follows from Theorem~\ref{bt1} that if $0\in\sigma(T_{k}|\mathcal{M_{E}}),$ then there is $\lambda\in \sigma(\mathbb{M}_{T}|\mathcal{M_{E}})$ such that $0=\lambda(T_{k})=\tau_{k}+\beta_{k}A_{\lambda}$ for some $A_{\lambda}\in \mathbb{R}.$ By Lemma~\ref{colio} we have $\tau_{k+j} + \beta_{k+j}A_{\lambda}=0.$ Since $\tau_{j}\neq 0$ for all $j\in\mathbb{N}$ we must have $\tau_{k+j}=-\beta_{k+j}A_{\lambda}\neq 0$ for all $j\in\mathbb{N}.$ Hence
							\begin{equation}\label{trilby}
								\tau_{k+j} + \beta_{k+j}A_{\lambda}=\tau_{k+j} + \frac{\tau_{k+j}}{\tau_{k}}\beta_{k}A_{\lambda}
							\end{equation}
							giving $\beta_{k+j}=\frac{\tau_{k+j}}{\tau_{k}}\beta_{k}.$ Furthermore, the formula~\eqref{trilby} shows that $$\left(\frac{\tau_{k+j}}{\tau_{k}}I\right)T_{k}|_{\mathcal{M_{E}}}=T_{k+j }|_{\mathcal{M_{E}}}.$$ As also $(\theta_{T,k}^{*}T_{j}\theta_{T,k})T_{k}|_{\mathcal{M_{E}}}=T_{k+j }|_{\mathcal{M_{E}}}$ and the range of $T_{k}$ is dense in $\mathcal{M_{E}},$ we must have $$\theta_{T,k}^{*}T_{j}\theta_{T,k}|_{\mathcal{M_{E}}}=\frac{\tau_{j+k}}{\tau_{k}}I|_{\mathcal{M_{E}}}.$$
						\end{proof}

\section{Structure properties of injective half-centered operators}
						The aim of this section is to establish structure results for injective half-centered operators that satisfy the main assumptions: $\dim \E=1$ and $\mathcal{H}_{\E}=\mathcal{H}.$
						\\
						
						As it was mentioned after the statement of the main theorem, if $\dim\mathcal{M_{E}}= 1$ then $T$ is centered and moreover if $\bigvee_{k=0}^{\infty}T^{k}\E=\mathcal{H},$ then $T$ is a weighted shift. Hence in what follows, we assume that $\dim\mathcal{M_{E}}\geq 2.$
						\\
						
						First we discuss the spectrum of $\mathbb{M}_{T|\mathcal{H}_{1}} |\mathcal{M_{E}}\ominus \E.$ 
						
						\begin{prop}\label{total}
							If $\dim \mathcal{M_{E}}\geq 3$ then the spectrum of $\mathbb{M}_{T|\mathcal{H}_{1}} |\mathcal{M_{E}}\ominus\mathcal{E}$ contains at least two points.
						\end{prop}
						
						\begin{proof}
							As before, we denote by $P$ the orthogonal projection onto $\mathcal{H}_{1}=\overline{T\mathcal{H}}.$ To prove the the statement it is enough to see that if $\dim \mathcal{M_{E}}\geq 2$ and $PT_{k}\E\neq 0$ (such $k$ exists, otherwise $\dim \mathcal{M_{E}}=1$), then
							
							\begin{equation}\label{croshaw}
								\mathbb{M}_{T|\mathcal{H}_{1}}PT_{k}\E=\mathcal{M_{E}}\ominus \E.
							\end{equation}
							Since if $PT_{k}\E\in \mathcal{M_{E}}\ominus\mathcal{E}$ is a cyclic vector for $\mathbb{M}_{T|\mathcal{H}_{1}} |\mathcal{M_{E}}\ominus\mathcal{E},$ then the number of points in $\sigma(\mathbb{M}_{T|\mathcal{H}_{1}} |\mathcal{M_{E}}\ominus\mathcal{E})$ is equal to $\dim\mathcal{M_{E}}\ominus \E$ and by assumption, this number is larger than two.
							\\
							
							Let $A$ be the operator from Theorem~\ref{bt1}, then $PT_{k}\E=\beta_{k}P A\E$ so for any $j,k\in\mathbb{N}$ the two vectors
							$PT_{k}\E$ and $PT_{j}\E$ differ only by a constant multiple. Hence $PT_{j}\E\in \mathbb{M}_{T|\mathcal{H}_{1}}PT_{k}\E$ for any $j\in\mathbb{N}.$
							\\
							
							The space $\mathbb{M}_{T|\mathcal{H}_{1}}PT_{k}\E$ is of course a subspace of $\mathcal{M_{E}},$ so if we can prove that $(\mathbb{M}_{T|\mathcal{H}_{1}}PT_{k}\E)\oplus\E$ is invariant for every $T_{j},$ then since $\mathcal{M_{E}}$ is the smallest closed subspace containing $\E$ that is invariant under $\mathbb{M}_{T},$ this would imply $(\mathbb{M}_{T|\mathcal{H}_{1}}PT_{k}\E)\oplus\E=\mathcal{M_{E}}$ and therefore $\mathbb{M}_{T|\mathcal{H}_{1}}PT_{k}\E=\mathcal{M_{E}}\ominus \E.$ 
							\\
							
							So take any $x+c \E\in (\mathbb{M}_{T|\mathcal{H}_{1}}PT_{k}\E)\oplus\E$
							with $c\in\mathbb{C}$ and $x\in \mathbb{M}_{T|\mathcal{H}_{1}}PT_{k}\E.$ Then since $P+P_{\E}=I$ we have 
							$$T_{j}x+c T_{j}\E=(P+P_{\E})(T_{j}x+c T_{j}\E)=P T_{j}P x+P_{\E}T_{j}x+c P T_{j}\E+cP_{\E}T_{j}\E.$$
							As $PT_{j}P\in \mathbb{M}_{T|\mathcal{H}_{1}}$ and $P_{\E}T_{j}\E=\left\langle T_{j}\E,\E \right\rangle \E =\tau_{j}\E,$ we obtain
							
							$$(P T_{j}P x+c P T_{j}\E)+\left(\left\langle T_{j}x,\E\right\rangle +c\tau_{j}\right)\E\in\mathbb{M}_{T|\mathcal{H}_{1}}PT_{k}\E\oplus\E.$$
						\end{proof}
								
\subsection{Relating $\mathbb{M}_{T}|\mathcal{M_{E}}$ to $\mathbb{M}_{T|\mathcal{H}_{1}}|\mathcal{M_{E}}\ominus\mathcal{E};$ the discrete case}
	
	The purpose of the next two subsections is to show that when $\dim \mathcal{M_{E}} \geq 2,$ then there is a relation between the spectrum of $\mathbb{M}_{T}|\mathcal{M_{E}}$ and that of $\mathbb{M}_{T|\mathcal{H}_{1}} |\mathcal{M_{E}}\ominus\mathcal{E}.$ 
		\\
							
							To see where this relation comes from, assume for a moment that $\mathbb{M}_{T}|\mathcal{M_{E}}$ has an orthonormal basis of eigenvectors $x_{i}\in\mathcal{M_{E}}.$ Since $\bigvee_{k=0}^{\infty}T^{k}\mathcal{M_{E}}=\mathcal{H}$ and $\dim \E=1,$ we can find an eigenvector $x_{k}$ and a smallest integer $m\geq 1$ such that $T^{m}x_{k}$ is not orthogonal to $\mathcal{M_{E}}$ but $T^{j}x_{k}\bot \mathcal{M_{E}}$ for $1\leq j\leq m-1$ (if this set of $j$'s is non-empty!). Such $x_{k}$ and $m$ must exist; in fact the converse would imply $T\mathcal{H}\bot \mathcal{M_{E}}$ and hence $\mathcal{M_{E}}\subseteq \E,$ contradicting $\dim \mathcal{M_{E}}\geq 2$. 
							\\
							
							Now fix such $x_{k}$ and $m.$ From Proposition~\ref{jups} we get $T^{j}x_{k}\in V_{j}$ for $1\leq j\leq m-1.$ Moreover, $T^{j}x_{k}$ an eigenvector for $\mathbb{M}_{T}|V_{j}.$ This is due to the following calculation: given $l\in \mathbb{N}$ 
							$$T_{l}T^{j}x_{k}=P_{\mathcal{H}_{j}}T_{l}P_{\mathcal{H}_{j}}T^{j}x_{k}=$$
							$$\left(\theta_{T,j}\theta_{T,j}^{*}\right)T_{l}\left(\theta_{T,j}\theta_{T,j}^{*}\right)T^{j}x_{k}=\theta_{T,j}\left(\theta_{T,j}^{*}T_{l}\theta_{T,j}\right)\theta_{T,j}^{*}T^{j}x_{k}=$$
							$$\theta_{T,j}\left(\theta_{T,j}^{*}T_{l}\theta_{T,j}\right)T_{j}^{\frac{1}{2}}x_{k}=\theta_{T,j}T_{j}^{\frac{1}{2}}\left(\theta_{T,j}^{*}T_{l}\theta_{T,j}\right)x_{k}=\lambda\left(\theta_{T,j}^{*}T_{l}\theta_{T,j}\right)T^{j}x_{k}$$ where $\lambda \in \sigma (\mathbb{M_{T}})$ is the eigenvalue corresponding to $x_{k}.$ 
							\\
							
							Next we observe that $T^{m}x_{k}$ can not be an eigenvector for $\mathbb{M}_{T}.$
							In fact, assuming contrary that $b T^{m}x_{k}=\gamma(b)T^{m}x_{k}$ for all $b\in \mathbb{M}_{T},$ where $\gamma \in\sigma(\mathbb{M}_{T}).$ But then as $T^{m}x_{k}\bot \E,$ we obtain 
							$$\left\langle T^{m}x_{k},b\mathcal{E} \right\rangle=\left\langle b T^{m}x_{k},\mathcal{E} \right\rangle=\gamma\left(b\right)\left\langle  T^{m}x_{k},\mathcal{E} \right\rangle=0$$
							as $T^{m}x_{k} \bot \E$ and hence $T^{m}x_{k}\bot \mathcal{M_{\E}}.$
							Therefore $T^{m}x_{k}$ can not be an eigenvector for $\mathbb{M}_{T}$, since it is orthogonal to $\mathcal{E}$ but not to $\mathcal{M_{E}}.$ 
							\\
							
							However, $T^{m}x_{k}$ must be an eigenvector for $\mathbb{M}_{T|\mathcal{H}_{1}}$ since $$PT_{l}PT^{m}x_{k}=\left(\theta_{T}\theta_{T}^{*}\right)T_{l}\left(\theta_{T}\theta_{T}^{*}\right)T^{m}x_{k}=\theta_{T}\left(\theta_{T}^{*}T_{l}\theta_{T}\right)T_{1}^{\frac{1}{2}}T^{m-1}x_{k}=$$ $$\gamma_{m-1}\left(\theta_{T}^{*}T_{l}\theta_{T}\right)T^{m}x_{k}=\lambda\left(\theta_{T,m}^{*}T_{l}\theta_{T,m}\right)T^{m}x_{k}$$ where $\gamma_{m-1}\in \sigma (\mathbb{M}_{T}|V_{m-1})$ is the eigenvalue corresponding to $T^{m-1}x_{k}\in V_{k-1}$ (the last equality follows from Proposition~\ref{loveu}). If we project $T^{m}x_{k}$ onto $\mathcal{M_{E}},$ then this will still be an eigenvector, since the projection commutes with $\mathbb{M}_{T|\mathcal{H}_{1}}$. 
							\\
							
							From this we see that for one of the points $\gamma$ in the spectrum of $\mathbb{M}_{T|\mathcal{H}_{1}} |\mathcal{M_{E}}\ominus\mathcal{E}$ there is $\lambda$ in the spectrum of $\mathbb{M}_{T} |\mathcal{M_{E}}$ is such that 
							\begin{equation}\label{mila1}
								\gamma\left(PT_{l}P\right)=\lambda\left(\theta_{T,m}^{*}T_{l}\theta_{T,m}\right)
							\end{equation}
							for all $l\in\mathbb{N}.$ If we multiply both sides of~\eqref{mila1} with $\lambda\left(T_{m}\right)$ and use
							
							$\lambda\left(T_{m}\right)\lambda\left(\theta_{T,m}^{*}T_{l}\theta_{T,m}\right)=\lambda\left(T_{m+l}\right),$ we get the equality
							\begin{equation}\label{aase}
								\lambda\left(T_{m}\right)\gamma\left(PT_{l}P\right)=\lambda\left(T_{m+l}\right)
							\end{equation}
							which is valid for all $l\in\mathbb{N}.$
							This shows how it is possible to express some points in the spectrum of $\mathbb{M}_{T|\mathcal{H}_{1}} |\mathcal{M_{E}}\ominus \E$ in terms of the spectrum of $\mathbb{M}_{T}|\mathcal{M_{E}}.$
							pAxa
\subsection{Relating $\mathbb{M}_{T}|\mathcal{M_{E}}$ to $\mathbb{M}_{T|\mathcal{H}_{1}} |\mathcal{M_{E}}\ominus\mathcal{E};$ the general case}
								
								A similar reasoning as used to derive~\eqref{mila1} can be generalized to work even in the general case, but due to the possible lack of eigenvectors, the proof of Proposition~\ref{aloha} uses the above arguments in a "reversed" way. However, this approach has a disadvantage of making less clear what the central idea is. This is why we included the discrete case as motivation.
								\\
								
								First we need an easy result.
								
								\begin{lem}\label{ludde}
									There is an isomorphism
									$$\Psi:\mathbb{M}_{T|\mathcal{H}_{1}} |\mathcal{M_{E}}\ominus \E \rightarrow \mathbb{M}_{T}^{1} |\theta_{T}^{*}\mathcal{M_{E}}$$
									induced by $$b\in \mathbb{M}_{T|\mathcal{H}_{1}} |\mathcal{M_{E}}\mapsto \theta_{T}^{*} b \theta_{T}|\theta_{T}^{*}\mathcal{M_{E}}.$$
								\end{lem}
								
								We can now proceed to prove the generalization of the result in the last subsection to the case when we may not have any non-trivial eigenvectors of $\mathbb{M}_{T}.$
								
								\begin{prop}\label{aloha}
									If $\bigvee_{k=0}^{\infty}T^{k}\mathcal{M_{E}}=\mathcal{H},$ then there is a dense subset of the spectrum of $\mathbb{M}_{T|\mathcal{H}_{1}} |\mathcal{M_{E}}\ominus \E$ such that for every $\gamma\in M$ there is a point $\lambda$ in the spectrum of $\mathbb{M}_{T}|\mathcal{M_{E}}$ and an integer $m\in\mathbb{N}$ such that $$\gamma\left(PT_{k}P\right)=\lambda\left(\theta_{T,m}^{*}T_{k}\theta_{T,m}\right)$$ for all $k\in\mathbb{N}.$
								\end{prop}

								\begin{proof}
									Consider the subspace $\theta_{T}^{*}\mathcal{M_{E}}.$ Since $\dim\mathcal{M_{E}}\geq 2$ and $\dim\ker \theta_{T}^{*}=1,$ this subspace is nonzero.
									By Theorem~\ref{space1} the projections $P_{V_{k}}$ adds up to the identity, so there exists a $m\in\mathbb{N}$ such that $P_{V_{m}}\theta_{T}^{*}\mathcal{M_{E}}\neq 0.$
									Since the projection $P_{V_{k}}$ commutes with $\mathbb{M}_{T},$ we have a homomorphism $$s_{k}:\mathbb{M}_{T|\mathcal{H}_{1}} |\mathcal{M_{E}}\ominus\mathcal{E}\rightarrow \mathbb{M}_{T}^{1}|P_{V_{k}}\theta_{T}^{*}\mathcal{M_{E}}$$ that is defined as the composition $$\mathbb{M}_{T|\mathcal{H}_{1}} |\mathcal{M_{E}}\ominus\mathcal{E}\stackrel{\Psi}{\rightarrow}  \mathbb{M}_{T}^{1} |\theta_{T}^{*}\mathcal{M_{E}}\rightarrow \mathbb{M}_{T}^{1} |P_{V_{k}}\theta_{T}^{*}\mathcal{M_{E}}$$ where $\Psi$ is the isomorphism from Lemma~\ref{ludde} and the second arrow is the restriction. The homomorphism $s_{k}$ induce an injective continuous map $$s_{k}^{*}:\sigma(\mathbb{M}_{T}^{1} |P_{V_{k}}\theta_{T}^{*}\mathcal{M_{E}})\rightarrow\sigma( \mathbb{M}_{T|\mathcal{H}_{1}} |\mathcal{M_{E}}\ominus\mathcal{E}).$$ As $\sum P_{V_{k}}=I$ and $\Psi$ is an isomorphism, given $a\in\mathbb{M}_{T|\mathcal{H}_{1}} |\mathcal{M_{E}}\ominus\mathcal{E} ,$ we have $a=0$ iff $s_{k}\left(a\right)=0$ for all $k\in\mathbb{N}.$ So the union of the ranges of all $s_{k}^{*}$ must be dense in $\sigma( \mathbb{M}_{T|\mathcal{H}_{1}} |\mathcal{M_{E}}\ominus\mathcal{E}).$
									\\
									
									If $\mu\in\sigma(\mathbb{M}_{T}^{1} |P_{V_{k}}\theta_{T}^{*}\mathcal{M_{E}})$ then there is a $\mu_{k}\in\sigma(\mathbb{M}_{T} |V_{k})$ such that $$\mu\left(\theta_{T}^{*}T_{j}\theta_{T}\right)=\mu_{k}\left(\theta_{T}^{*}T_{j}\theta_{T}\right)$$ and so from Proposition~\ref{loveu} there is a $\lambda\in \mathbb{M}_{T}|\mathcal{M_{E}}$ such that $$\mu\left(\theta_{T}^{*}T_{j}\theta_{T}\right)=\mu_{k}\left(\theta_{T}^{*}T_{j}\theta_{T}\right)=\lambda\left(\theta_{T,k+1}T_{j}\theta_{T,k+1}\right).$$
									Taking $\gamma=s^{*}_{k}(\mu),$ we have $\gamma\left(PT_{j}P\right)=\mu\left(\theta_{T}^{*}T_{j}\theta_{T}\right)$ and so $$\gamma\left(PT_{j}P\right)=\mu\left(\theta_{T}^{*}T_{j}\theta_{T}\right)=\lambda\left(\theta_{T,k+1}T_{j}\theta_{T,k+1}\right)$$ for all $j\in \mathbb{N}.$
									Now take $m=k+1.$
								\end{proof}
								
								Proposition~\ref{aloha} motivates the following definition:
								
								\begin{defn}
									Let $\mathcal{F}$ be the set of all triples $\left(\lambda,\gamma,m\right)$ consisting of $$\lambda\in \sigma( \mathbb{M}_{T} |\mathcal{M_{E}})$$ $$\gamma\in \sigma( \mathbb{M}_{T|\mathcal{H}_{1}} |\mathcal{M_{E}}\ominus\mathcal{E})$$ and a $m\in\mathbb{N}^{+}$ such that
									$$\gamma\left(PT_{k}P\right)=\lambda\left(\theta_{T,m}^{*}T_{k}\theta_{T,m}\right)$$ for all $k\in\mathbb{N}.$
									We say that the triples $\left(\lambda,\gamma,m\right)$ and $\left(\lambda',\gamma',m'\right)$ are equal if $\lambda\neq\lambda'$ or $\gamma\neq\gamma'$ or $m\neq m'.$ 
									
								\end{defn}
								
								Recall from Theorem~\ref{bt1} that if $\lambda\in\sigma( \mathbb{M}_{T} |\mathcal{M_{E}})$ and $\gamma\in \sigma( \mathbb{M}_{T|\mathcal{H}_{1}} |\mathcal{M_{E}}\ominus\mathcal{E})$ then 
								$$\lambda(T_{k})=\tau_{k}+\beta_{k}A_{\lambda}$$
								$$\gamma(P T_{k}P)=\tau_{k}+\beta_{k}C_{\gamma}$$
								for some $A_{\lambda},C_{\gamma}\in\mathbb{R}.$
								\\
								
								Next proposition shows how every triple $(\lambda,\gamma,m)\in\mathcal{F}$ gives rise to a relation between the $\tau_{k}$'s and  $\beta_{k}$'s.
								
								\begin{prop}\label{olov}\label{ericagis}
									For any triple $\left(\lambda,\gamma,m\right)\in\mathcal{F}$ and every $k\in\mathbb{N}$ we have 
									\begin{equation}\label{oleg}
										\lambda\left(T_{m}\right)\gamma\left(PT_{k}P\right)=\lambda\left(T_{m+k}\right).
									\end{equation}
									
									Moreover, if $\lambda\left(T_{m}\right)=\tau_{m}+A_{\lambda}\beta_{m}$ and $\gamma\left(PT_{m}P\right)=\tau_{m}+C_{\gamma}\beta_{m}$ then for all $k\in\mathbb{N}$
									\begin{equation}\label{mars}
										\tau_{k}-\frac{\tau_{m+k}}{\lambda\left(T_{m}\right)}=\frac{A_{\lambda}\beta_{m+k}}{\lambda\left(T_{m}\right)}-C_{\gamma}\beta_{k}
									\end{equation}
									when $\lambda\left(T_{m}\right)\neq 0$ and 
									
									\begin{equation}\label{30smars}
										\tau_{k}-\frac{\tau_{m+k}}{\tau_{m}}=-C_{\gamma}\beta_{k}
									\end{equation}
									when $\lambda\left(T_{k}\right)=0.$
									
								\end{prop}
								
								\begin{proof}
									We have $\gamma\left(PT_{k}P\right)=\lambda\left(\theta_{T,m}^{*}T_{k}\theta_{T,m}\right),$
									so $$\lambda\left(T_{m}\right)\gamma\left(PT_{k}P\right)=\lambda\left(T_{m}\right)\lambda\left(\theta_{T,m}^{*}T_{k}\theta_{T,m}\right)=\lambda\left(T_{m+k}\right)$$ proving the first part.
									If $\lambda\left(T_{m}\right)\neq 0,$ then
									$$\lambda\left(T_{m}\right)\tau_{k}+\lambda\left(T_{m}\right)C_{\gamma}\beta_{k}=\lambda\left(T_{m}\right)\gamma\left(PT_{k}P\right)=\lambda\left(T_{m+k}\right)$$
									by~\eqref{oleg}.
									As $\lambda\left(T_{m+k}\right)=\tau_{m+k}+A_{\lambda}\beta_{m+k},$ we obtain \eqref{mars}.
									When $\lambda\left(T_{m}\right)= 0,$ we get the formula from Propositions~\ref{loveu} and~\ref{queen}.
								\end{proof}
								
\section{Main theorem: The case $|\mathcal{F}|\geq 2$}
								The aim of this section is to show that when $\mathcal{F}$ has at least two elements, then $T$ satisfies equation~\eqref{gabe} in the main theorem. 
								\\
								
								Let $\left\lbrace \tau_{k}\right\rbrace $ and $\left\lbrace \beta_{k}\right\rbrace$
								be the sequences of real numbers associated to $T$ that are defined by~\eqref{tau} and~\eqref{bee}.
								Let
								
								$$\tau\left(z\right)=\sum_{j=0}^{\infty}\tau_{j}z^{j}$$ 
								and
								
								$$B\left(z\right)=\sum_{j=0}^{\infty}\beta_{j}z^{j}.$$ 
								be formal powerseries associated to $\left\lbrace \tau_{k}\right\rbrace $ and $\left\lbrace \beta_{k}\right\rbrace.$
								\\
								
								Let $S^{*}$ be the backwards shift operator, defined on power series as $$S^{*}:\sum_{k=0}^{\infty}a_{k}z^{k}\mapsto\sum_{k=0}^{\infty}a_{k+1}z^{k}$$ and pick $\left(\lambda,\gamma,m\right)\in\mathcal{F}.$ Then~\eqref{mars} and~\eqref{30smars} can be rewritten as follows:
								
								\begin{equation}\label{evil}
									\left( I-\frac{S^{*m}}{\lambda\left(T_{m}\right)} \right)\tau\left(z\right)=-\left(C_{\gamma} I-\frac{A_{\lambda} S^{*m}}{\lambda\left(T_{m}\right)}\right)B\left(z\right).
								\end{equation} 
								when $\lambda\left(T_{m}\right)\neq 0$ and
								
								\begin{equation}\label{evil1}
									\left( I-\frac{S^{*m}}{\tau_{m}} \right)\tau\left(z\right)=-C_{\gamma}B\left(z\right)
								\end{equation}
								otherwise. 
								\\
								
								Taking another triple $(\mu,\omega,n)\in \mathcal{F}$ we obtain similar equalities with $(\lambda, \gamma, m)$ replaced by $(\mu, \omega, n)$
								\\

								Letting
								\\
								\begin{center}
									$P_{1}(z) = \begin{cases}
									1-\frac{z^{m}}{\lambda(T_{m})}          & \lambda(T_{m}) \neq 0\\
									1-\frac{z^{m}}{\tau_{m}} & \text{otherwise}
									\end{cases},$
									$P_{2}(z) =\begin{cases}
									C_{\gamma}-\frac{A_{\lambda}z^{m}}{\lambda(T_{m})}          & \lambda(T_{m}) \neq 0\\
									C_{\gamma} & \text{otherwise}
									\end{cases},$ 
								\end{center}
								\begin{center}

									$Q_{1}(z) = \begin{cases}
									1-\frac{z^{n}}{\mu(T_{n})}          & \mu(T_{n}) \neq 0\\
									1-\frac{z^{n}}{\tau_{n}} & \text{otherwise}
									\end{cases}$
									$Q_{2}(z) = \begin{cases}
									C_{\omega}-\frac{A_{\mu}z^{n}}{\mu(T_{n})}           & \mu(T_{m}) \neq 0\\
									C_{\omega} & \text{otherwise}
									\end{cases}.$
								\end{center}
								We have 
								
								\begin{equation}\label{bobby}
									\begin{split}
										P_{1}(S^{*})\tau(z)=-P_{2}(S^{*})\beta(z) \\
										Q_{1}(S^{*})\tau(z)=-Q_{2}(S^{*})\beta(z).
									\end{split}
								\end{equation}
								Now let $P(z)=P_{1}(z)Q_{2}(z)-P_{2}(z)Q_{1}(z).$
								
								\begin{lem}
									We have 
									\begin{equation*}
										\begin{split}
											P(S^{*})\tau(z)=0\\
											P(S^{*})\beta(z)=0
										\end{split}
									\end{equation*}
								\end{lem}
								
								\begin{proof}
									If follows from~\eqref{bobby} that 
									$$P_{1}(S^{*})Q_{1}(S^{*})\tau(z)=-P_{1}(S^{*})Q_{2}(S^{*})\beta(z)$$
									$$Q_{1}(S^{*})P_{1}(S^{*})\tau(z)=-Q_{1}(S^{*})P_{2}(S^{*})\beta(z)$$
									giving the first equality $P(S^{*})\beta(z)=0.$
									A similar calculation, gives 
									\\ $P(S^{*})\tau(z)=0.$
								\end{proof}
								
								Our next goal is to show that if $\mathcal{F}$ contains at least two triples we can choose $(\lambda,\gamma,m)$ and $(\mu,\omega,n)$ such that $P(z)$ is not identically zero.
								\\
								
								Since we will always work with only two triples at the time, we can without any resulting confusion denote the polynomial corresponding to
								\\ $\left(\lambda_,\gamma,m\right),\left(\mu,\omega,n\right)$ by $P\left(z\right).$
								\begin{lem}\label{terrible}
									If $\dim\mathcal{M_{E}}= 2$ then for every triple $\left(\lambda,\gamma,m\right)\in\mathcal{F}$ we have $A_{\lambda}\neq C_{\gamma}.$ 
								\end{lem}
								
								\begin{proof} 
									Choose $k\in\mathbb{N}$ such that $\beta_{k}\neq 0.$ Then the spectrum of $T_{k}|\mathcal{M_{E}}$ consists of two points, since from~\eqref{raw} we find that in this case, $I$ and $T_{k}$ are generators for $\mathbb{M}_{T}|\mathcal{M_{E}}$. Now consider the function $H\left(z\right)=\left\langle (T_{k}-z)^{-1}\mathcal{E},\mathcal{E}\right\rangle.$ This is a rational function with single poles at the eigenvalues of $T_{k}.$ Since for real $z$ $$H'\left(z\right)=\left\langle (T_{k}-z)^{-2}\mathcal{E},\mathcal{E}\right\rangle=\left\langle (T_{k}-z)^{-1}\mathcal{E},(T_{k}-z)^{-1}\mathcal{E}\right\rangle>0$$ we see that $H\left(z\right)$ has a zero $\chi$ between its two poles. As $\chi$ is not in the spectrum of $T_{k}$ we must have $(T_{k}-\chi)^{-1}\mathcal{E}\neq 0$ and then from $\left\langle (T_{k}-\chi)^{-1}\mathcal{E},\mathcal{E}\right\rangle=0$ we get $(T_{k}-\chi)^{-1}\mathcal{E}\bot \mathcal{E}.$ Now we can calculate $$PT_{k}P(T_{k}-\chi)^{-1}\mathcal{E}=P T_{k}(T_{k}-\chi)^{-1}\mathcal{E}=$$ $$P(T_{k}-\chi+\chi)(T_{k}-\chi)^{-1}\mathcal{E}=\chi(T_{k}-\chi)^{-1}\mathcal{E}.$$ Hence $\chi$ is in the spectrum of $PT_{k}P|\mathcal{M_{E}}\ominus \E $ so $$\chi=\tau_{k}+C_{\gamma}\beta_{k}.$$ But $\chi$ is not in the spectrum of $T_{k}|\mathcal{M_{E}}$ and therefore $\chi\neq\tau_{k}+\beta_{k}A_{\lambda}.$ As $\beta_{k}\neq 0,$ we get $C_{\gamma}\neq A_{\lambda}.$
								\end{proof}
								
								If $\dim\mathcal{M_{E}}= 2,$ there is only one element in $\sigma( \mathbb{M}_{T|\mathcal{H}_{1}} |\mathcal{M_{E}}\ominus\mathcal{E}).$ Hence for two triples $\left(\lambda,\gamma,m\right),\left(\mu,\omega,n\right)\in\mathcal{F}$ we must have $\gamma=\omega.$ 
								
								\begin{lem}\label{labam}
									Let $\dim\mathcal{M_{E}}=2.$ If there are two different triples $\left(\lambda,\gamma,m\right),\left(\mu,\gamma,n\right)\in\mathcal{F},$ then the polynomial $P(z)$ is not constantly zero. However, we have $P(0)=0.$
								\end{lem}
								
								\begin{proof}
									First, note that since $\dim\mathcal{M_{E}}=2$ and $T$ is injective we can not have $\lambda(T_{k})=0$ for any $\lambda\in\sigma(\mathbb{M}_{T}|\mathcal{M_{E}})$ and $k\in\mathbb{N},$ since otherwise we will have a nonzero $u\in\mathcal{M_{E}}$ such that $T_{k}u=0$ and hence $0=\left\langle T_{k}u,u \right\rangle =\left\|T^{k}u  \right\| .$ Let $\left(\lambda,\gamma,m\right),\left(\mu,\gamma,n\right)\in\mathcal{F},$ then $\lambda(T_{m})\neq 0$ and $\mu(T_{n})\neq 0.$ The corresponding polynomial $P(z)$ is then of the form 
									\begin{equation}\label{rere}
										P\left(z\right)=\left(C_{\gamma}-\frac{A_{\lambda}z^{m} }{\lambda\left(T_{m}\right)} \right)\left( 1-\frac{z^{n}}{\mu\left(T_{n}\right)} \right)-\left(C_{\gamma}-\frac{A_{\mu}z^{n}}{\mu\left(T_{n}\right)} \right)\left( 1-\frac{z^{m}}{\lambda\left(T_{m}\right)} \right).
									\end{equation}
									Assume on the contrary that $P(z)\equiv 0.$ By expanding the right-hand side of~\eqref{rere} and use $A_{\lambda}\neq C_{\gamma}$ and $A_{\mu}\neq C_{\gamma},$ we easily see that $P(z)\equiv 0$ implies $m=n$ and $A_{\lambda}=A_{\mu}$ and hence the triples are equal. The second claim follows from that the constant term in the right-hand side of~\eqref{rere} vanishes.
								\end{proof}
								
								\begin{prop}\label{arshole}\label{movie}
									If the set $\mathcal{F}$ has more than two elements, then there are two triples $\left(\lambda_,\gamma,k\right),\left(\mu,\omega,m\right)\in\mathcal{F}$ such that the polynomial $P(z)$ is not the zero polynomial.
									Moreover, if $\dim\mathcal{M}\geq 3,$ then there are different triples such that $P\left(0\right)\neq 0$.
								\end{prop}
								\begin{proof}
									
									We already know that if $\dim \mathcal{M_{E}}=2$ and there are two different triples, then the polynomial $P(z)$ is not constantly zero.
									If $\dim\mathcal{M_{E}}\geq 3$ then by Lemma~\ref{total} the algebra generated by the $P T_{j} P$'s restricted to $\mathcal{M_{E}}$ must have a spectrum consisting of at least $2$ different points. Proposition~\ref{aloha} now gives that there are $\gamma,\omega\in M$ with $\gamma\neq \omega$ and thus also with $C_{\gamma}\neq C_{\omega}.$ An easy calculation gives that the constant term of $P(z)$ is   $C_{\gamma}-C_{\omega}$ and hence $P\left(z\right)\neq 0.$ 
								\end{proof}
								
								Now we can prove the main result of this section:
								\begin{thm}\label{mainmain}
									If $\mathcal{F}$ has at least two elements, then there are constants $a,b,c,d\in\mathbb{R},$ not all zero and integers $n,m\in\mathbb{N}^{+}$, such that 
									\begin{equation}\label{notinsert}
										a I+b T_{n}+c T_{m}+d T_{n+m}=0.
									\end{equation}
									In particular, if $\dim \mathcal{M_{E}}\geq 3$ then we can choose~\eqref{notinsert} such that $a\neq 0.$
								\end{thm}
								
								\begin{proof}
									If $\mathcal{F}$ has at least two elements, it follows from Proposition~\ref{movie} that there exists $\left(\lambda_,\gamma,k\right)$ and $\left(\mu,\omega,m\right)$ in $\mathcal{F}$ such that the corresponding polynomial $P(z)$ is of the form $a+b z^{n}+ c z^{m}+d z^{n+m},$ where $a,b,c,d\in\mathbb{R}$ are not all zero and $n,m\in\mathbb{N}^{+}.$
									As $P\left(S^{*}\right)B\left(z\right)=0$ and $P\left(S^{*}\right)\tau\left(z\right)=0,$ we obtain that for all $k\in\mathbb{N}$
									\begin{equation}\label{refs}
										a \tau_{k}+b \tau_{k+n}+c \tau_{k+m}+d \tau_{k+n+m}=0
									\end{equation}
									\begin{equation}\label{refs1}
										a \beta_{k}+b \beta_{k+n}+c \beta_{k+m}+d \beta_{k+n+m}=0.
									\end{equation}
									By Theorem~\ref{bt1}, these equations implies
									$$aT_{k}+bT_{n+k}+ c T_{m+k}+ d T_{n+m+k}|_{\mathcal{M_{E}}}=0|_{\mathcal{M_{E}}}$$
									for all $k\in\mathbb{N}.$ Now fix $k\in\mathbb{N}$ and consider $$aI+b(\theta_{T,k}^{*}T_{n}\theta_{T,k})+ c (\theta_{T,k}^{*}T_{m}\theta_{T,k})+ d (\theta_{T,k}^{*}T_{n+m+k}\theta_{T,k}).$$ Restricted to $\mathcal{M_{E}}$, we have 
									$$T_{k}^{\frac{1}{2}}\left(aI+b(\theta_{T,k}^{*}T_{n}\theta_{T,k})+ c (\theta_{T,k}^{*}T_{m}\theta_{T,k})+ d (\theta_{T,k}^{*}T_{n+m}\theta_{T,k}^{*})\right)T_{k}^{\frac{1}{2}}|_{\mathcal{M_{E}}}=$$
									$$aT_{k}+bT_{n+k}+ c T_{m+k}+ d T_{n+m+k}|_{\mathcal{M_{E}}}=0|_{\mathcal{M_{E}}}.$$
									Since $T_{k}^{\frac{1}{2}}$ has dense range, we must have $$a I+b(\theta_{T,k}^{*}T_{n}\theta_{T,k})+ c (\theta_{T,k}^{*}T_{m}\theta_{T,k})+ d( \theta_{T,k}^{*}T_{n+m}\theta_{T,k}^{*})|_{\mathcal{M_{E}}}=0|_{\mathcal{M_{E}}}.$$
									By Theorem~\ref{normspos}, this implies that
									$$aI+bT_{n}+ c T_{m}+ d T_{n+m}|_{V_{k}}=0|_{V_{k}}.$$
									As this is true for every $k\in\mathbb{N}$ and the subspaces $V_{k}$ spans $\mathcal{H}$, we have 
									$$aI+bT_{n}+ c T_{m}+ d T_{n+m}=0.$$
									If $\dim \mathcal{M_{E}}\geq 3 $ then by Proposition~\ref{movie} there are $\left(\lambda,\gamma,n\right),\left(\mu,\omega,m\right) \in\mathcal{F}$ such that $P(0)=a\neq 0.$
									
								\end{proof}
								
								\begin{cor}\label{xxxx}
									For all $k,j\in\mathbb{N},$ the restriction $T_{j}|V_{k}$ is invertible. If $\dim \mathcal{M_{E}}\geq 3,$ then $T_{j}$ is invertible for all $j\in\mathbb{N}$, or equivalently, $T$ has closed range. 
								\end{cor}
								\begin{proof}
									Since $T$ is injective, every $T_{j}$ has dense range. If $\dim \mathcal{M_{E}}\leq 2$ then $\dim V_{k}\leq 2$ for all $k\in\mathbb{N},$ so $T_{j}|V_{k}$ must be invertible. 
									\\
									
									When $\dim \mathcal{M_{E}}\geq 3,$ it follows from Theorem~\ref{mainmain} that there are $b,c,d\in\mathbb{R}$ and $m,n\in\mathbb{N}^{+}$ such that $$I+bT_{n}+cT_{m}+dT_{n+m}=0. $$ (we can divide~\eqref{notinsert} by $a\neq 0$). If, say, $n\leq m$ then consider $$-bI-c (\theta_{T,n}^{*}T_{m-n}\theta_{T,n})-d(\theta_{T,n}^{*}T_{m}\theta_{T,n}).$$ This is an inverse of $T_{n}$ since 
									$$T_{n}\left(-b I-c (\theta_{T,n}^{*}T_{m-n}\theta_{T,n})-d (\theta_{T,n}^{*}T_{m}\theta_{T,n})\right)=-b T_{n}-c T_{m}-d T_{n+m}=I.$$ But if $T_{n}$ is invertible, then so is $T_{1}$, since $T_{n}=T_{1}(\theta_{T}^{*}T_{n-1}\theta_{T}).$ 
								\end{proof}

							\subsection{Main theorem: The case $|\mathcal{F}|=1$}
								
								The final case to consider is when there is only one triple in $\mathcal{F}.$ 
								\\
								
								Take $J$ to be a weighted shift on $\ell^{2}$ with the standard basis $\left\{e_{k};k\in\mathbb{N}\right\}.$ Now for some $n\in\mathbb{N}$ and $a\in\mathbb{C}$ consider $$L=J+a (e_{0}\otimes e_{n}^{*})$$ (as in Example~\ref{llll}, $e_{0}\otimes e_{n}^{*}$ is the operator $x\mapsto \left\langle x,e_{n} \right\rangle e_{0}$). With respect to the standard basis, this infinite matrix will look as follows
								\begin{equation}\label{dia0}
									L=\left[ \begin{matrix}
										a & 0 & 0 & 0 & ... \\
										a_{0} & 0 & 0 & 0 & ... \\
										0 & a_{1} & 0 & 0 & ... \\
										0 & 0 & a_{2} & 0 & ... \\
										... & ... & ... & ... & ... \\
									\end{matrix} \right]
								\end{equation}
								when $n=0$ and 
								\begin{equation}\label{dian}
									L=\left[ \begin{matrix}
										0 & ... & a & 0 & ... \\
										a_{0} & ... & 0 & 0 & ... \\
										... & ... & ... & ... & ... \\
										0 & ... & a_{n} & 0 & ... \\
										... & ... & ... & ... & ... \\
									\end{matrix} \right]
								\end{equation}
								for a general $n.$
								\begin{lem}\label{mila}
									The operator $L$ is half-centered and for every $k\in\mathbb{N}$ we have that $L_{k}$ is diagonal with respect to the standard basis $\left\{e_{k};k\in\mathbb{N}\right\}.$
								\end{lem}
								\begin{proof}
									As it was mention after Example~\ref{llll}, this is a corollary of Proposition~\ref{exempl}.
								\end{proof}
								
								In this section, we prove the following:
								\begin{thm}\label{semidia}
									If $\mathcal{F}$ has only one triple $\left(\lambda,\gamma,n\right),$ then there is an orthonormal basis $\left\{x_{k};k\in\mathbb{N}\right\}$ of $\mathcal{H},$ a wighted shift $J$ on this basis and $a\in\mathbb{C}$ such that $$T=J+a (x_{0}\otimes x_{n}^{*}).$$
								\end{thm}
								
								There is an orthonormal basis $v,w$ of $\in\mathcal{M_{E}}$ consisting of common eigenvectors for all the $T_{j}'$s restricted to this space. Let, say, $w$ be an eigenvector corresponding to $\lambda.$ As there is only one triple, we must have $T^{k}v\in V_{k}$ for all $k\in\mathbb{N},$ otherwise the reasoning used in subsection 4.1 would yield another different triple. Furthermore:
								
								\begin{lem}
									If there is only one triple $\left(\lambda,\gamma,n\right)$ in $\mathcal{F},$ then $V_{n+j}=\left\langle T^{n+j}v \right\rangle $ for all $j\in\mathbb{N}.$
								\end{lem}
								
								\begin{proof}
									The reasoning used in the proof of Proposition~\ref{aloha} shows that the only way we could end up with only one triple $\left(\lambda,\gamma,n\right)$ is if $\dim \mathcal{M_{E}}=2$ and $T^{*}\mathcal{M_{E}}$ is a subspace of $V_{n-1}.$ This in turn gives $T V_{k}\bot \mathcal{M_{E}}$ for $k\neq n-1.$ 
									Hence $T V_{n+j}=V_{n+j+1}$ for all $j\in\mathbb{N}$ by Proposition~\ref{isisis} and Proposition~\ref{jups}. This shows that the subspace $\oplus_{m=n+2}^{\infty}V_{m}$ is $T$-invariant and so $\oplus_{m=0}^{n+1}V_{m}$ is $T^{*}$-invariant. Now $T^{*}V_{n+1}=T^{*}T V_{n}=V_{n}$ and $$\left\langle T^{*} V_{j}, V_{n+1} \right\rangle=\left\langle  V_{j},T V_{n+1} \right\rangle=\left\langle  V_{j}, V_{n+2} \right\rangle=0$$ for $j\neq n+2,$ so $V_{n+1}\bot T^{*}\oplus_{m=0}^{n+1}V_{m}.$ But $T^{*}$ restricted to $\oplus_{m=0}^{n+1}V_{m}$ still has just $\mathcal{E}$ as its kernel and since the space $\oplus_{m=0}^{n+1}V_{m}$ has finite dimension, the dimension of the kernel must be equal that of the cokernel. So $\dim V_{n+1}=1$ and since $V_{n+1}=T V_{n}$ this must also be true for $V_{n}.$ Since $T^{n+j}v\in V_{n+j},$ the whole space must be spanned by this vector.
								\end{proof}
								
								\begin{prop}
									We have $T^{n}w\in \mathcal{M_{E}}.$
								\end{prop}
								
							\begin{proof}
							If $m\geq n+1,$ then as $T^{n}w\in X_{n},$ we have $T^{n}w\bot V_{m}$ by the definition of $V_{m}$. Also $V_{n}=\left<T^{n}v\right>$ and so $$\left\langle T^{n}v,T^{n}w\right\rangle=\left\langle T_{n}v,w\right\rangle=\lambda(T_{n})\left\langle v,w\right\rangle=0.$$ The same argument shows that $T^{n}w\bot V_{m}$ for $ 1\leq m\leq n-1$ since 
									$$\left\langle T^{m}w,T^{n}w\right\rangle=\left\langle T_{m}w,T^{n-m}w\right\rangle=0$$
									$$\left\langle T^{m}v,T^{n}w\right\rangle=\left\langle T_{m}e_{0},T^{n-m}w\right\rangle=0$$ 
									and these vectors span $V_{m}$ for $ 1\leq m\leq n-1.$ So $T^{n}w\in \left(\oplus_{k=1}^{\infty}V_{k}\right)^{\bot}=\mathcal{M_{E}}.$
								\end{proof}
								
								\begin{cor}
									We have $\mathcal{M_{E}}\ominus \mathcal{E}=\left<T^{n}w\right>$ and $T^{*}\mathcal{M_{E}}=\left< T^{n-1}w \right>.$
								\end{cor}
								\begin{proof}
									$T^{n}w\in\mathcal{M_{E}}$ is orthogonal to $\mathcal{E}$ and since $\dim \mathcal{M_{E}}=2,$ the subspace $\mathcal{M_{E}}\ominus \mathcal{E}$ must be generated by $T^{n}w.$ The second claim now follows from $$T^{*}\mathcal{M_{E}}=T^{*}\mathcal{M_{E}}\ominus \mathcal{E}=\left<T_{1} T^{n-1}w \right>= \left< T^{n-1}w \right>$$ since $T^{n-1}w$ is an eigenvector for $\mathbb{M}_{T}$ by the introduction to subsection 4.1.
								\end{proof}
								
								With the help of these result we can now proceed to prove Theorem~\ref{semidia}:
								
								\begin{proof}
									For $0\leq k\leq n-1$ take $$x_{k}=\frac{T^{k}w}{\left\|T^{k}w\right\|}$$ and when $n\leq k $ take $$x_{k}=\frac{T^{k-n}v}{\left\|T^{k-n}v\right\|}.$$ Thus $\left\{x_{k};k\in\mathbb{N}\right\}$ is an orthonormal basis for the Hilbert space $\mathcal{H}$ and by the results above, there is constants $a_{k},a\in \mathbb{C}$ such that
									$$T x_{k}=a_{k}x_{k+1}$$
									when $0\leq k\leq n-2$ or $n \leq j$ and $$T x_{n-1}=a_{n}x_{n}+a x_{0}$$ (since $T^{n}w$ was in $\mathcal{M_{E}}$ generated by $w=x_{0},v=x_{n}$). If we now take $J$ to be the shift $$J x_{k}=a_{k} x_{k+1}$$ then $$T=J+a (x_{0}\otimes x_{n}^{*}).$$
								\end{proof}
								We can now complete the proof of the main theorem, which we state again for references sake.
								\begin{mthm}
									Let $T$ be an injective half-centered operator on $\mathcal H$ such that $\bigvee_{k=0}^{\infty}T^{k}\mathcal{M_{E}}=\mathcal{H}$ and $\dim \left(T\mathcal{H}\right)^{\bot}=1.$ Then there are two not mutually exclusive possibilities.
									\renewcommand{\labelenumi}{\arabic{enumi}}
									\begin{enumerate}
										\item 
										There is an orthonormal basis $\left\{x_{k}:k\in\mathbb{N}\right\}$ of common eigenvectors for the operators $\left\{T_{k}\right\}_{k\in\mathbb{N}}$ such that with respect to this basis, $T$ is either a weighted shift or there is a weighted shift $J$ such that
										\begin{equation}\label{ebag1}
											T=J+a (x_{0}\otimes x_{n}^{*})
										\end{equation}
										for a $n\in\mathbb{N}$ and $a\in\mathbb{C}.$

										\item
										There are constants $a,b,c,d\in \mathbb{R},$ not all zero and $k,n\in \mathbb{N}^{+}$ such that 
										\begin{equation}\label{gabe1}
											aI+bT^{*k}T^{k} +c T^{*n}T^{n}+ d T^{*k+n}T^{k+n}=0.
										\end{equation}
										
									\end{enumerate}
									Moreover, if $\dim\mathcal{M_{E}}\geq 3$ then~\eqref{gabe1} holds with $a\neq 0$ and the range of $T$ is closed.
								\end{mthm}
								
								\begin{proof}
									When $\dim \mathcal{M_{E}}=1,$ we refer to the remarks given after the statement of the main theorem in section $2.$ When $\dim\mathcal{M_{E}}\geq 2,$ it follows from Propositions~\ref{total} and~\ref{aloha} that $\left|\mathcal{F} \right|\geq 1$ and hence we can split the argument into the cases $\left|\mathcal{F} \right|= 1$ and $\left|\mathcal{F} \right|\geq 2.$ When $\left|\mathcal{F} \right|= 1,$ we get from Theorem~\ref{semidia} that this corresponds to the second part of case \emph{1} above. When $\left|\mathcal{F} \right|\geq 2,$ we get~\eqref{gabe1} from Theorem~\ref{mainmain}. Finally, when $\dim\mathcal{M_{E}}\geq 3,$ the claim follows from Theorem~\eqref{mainmain} and Corollary~\ref{xxxx}.
								\end{proof}
							
							{}

\begin{thebibliography}{}
								
								
								
								\bibitem{CO1}
								Agler, Jim; Stankus, Mark. m-isometric transformations of Hilbert space. I. Integral Equations Operator Theory 21 (1995), no. 4, 383-429.  
								
								\bibitem{CO2}
								Agler, Jim; Stankus, Mark. m-isometric transformations of Hilbert space. II. Integral Equations Operator Theory 23 (1995), no. 1, 1-48. 
								\bibitem{CO3}
								Agler, Jim; Stankus, Mark. m-isometric transformations of Hilbert space. III. Integral Equations Operator Theory 24 (1996), no. 4, 379-421. 
								
								\bibitem{ARS}
								Morrel, Bernard; Muhly, Paul. Centered operators. Studia Mathematica 51 (1974), 251-263. 
								
								\bibitem{muphy}
								Murphy, Gerard J. $C^{*}$-Algebras and Operator Theory. Academic Press, Inc., Boston, MA, 1990.
								
								\bibitem{vp}
								V.Paulsen, C.Pearcy and S.Petrović. On centered and weakly centered operators. J.Funct. Anal., 128(1995), 87–101.
								
								\bibitem{SS}
								Shimorin, Serguei, Wold-type decompositions and wandering subspaces for operators close to isometries. J. Reine Angew. Math. 531 (2001), 147-189.
							\end{thebibliography}
\end{document}